\documentclass[reqno, 12pt]{amsart}

\usepackage{amsmath}
\usepackage{amsfonts}
\usepackage{amssymb}
\usepackage{mathtools}%for xmapsto
\usepackage{graphicx, color}
\usepackage{tikz}
\usetikzlibrary{cd, intersections, calc, decorations.pathmorphing, arrows, decorations.pathreplacing}
\usepackage[all]{xy}
\usepackage[abs]{overpic}

\usepackage{stackengine}
\usepackage{calc}
\newlength\shlength
\newcommand\xshlongvec[2][0]{\setlength\shlength{#1pt}%
  \stackengine{-5pt}{$#2$}{\smash{$\kern\shlength%
    \stackengine{7.1pt}{$\mathchar"017E$}%
      {\rule{\widthof{$#2$}}{.57pt}\kern.4pt}{O}{r}{F}{F}{L}\kern-\shlength$}}%
      {O}{c}{F}{T}{S}}

\usepackage{subfiles}
\usepackage[colorlinks, linkcolor=blue, citecolor=red, urlcolor=cyan]{hyperref}

\numberwithin{equation}{section}

\usepackage{cleveref}
\crefname{thm}{Theorem}{Theorems}
\crefname{cor}{Corollary}{Corollaries}
\crefname{lem}{Lemma}{Lemmas}
\crefname{sublem}{Sublemma}{Sublemmas}
\crefname{prop}{Proposition}{Propositions}
\crefname{dfn}{Definition}{Definitions}
\crefname{defi}{Definition}{Definitions}
\crefname{ex}{Example}{Examples}
\crefname{claim}{Claim}{Claims}
\crefname{conj}{Conjecture}{Conjectures}
\crefname{conv}{Notation}{Notations}
\crefname{rem}{Remark}{Remarks}
\crefname{rmk}{Remark}{Remarks}
\crefname{figure}{Figure}{Figures}
\crefname{section}{Section}{Sections}
\crefname{appendix}{Appendix}{Appendices}

\usepackage{amsthm}
\newtheorem{thm}{Theorem}[section]
\newtheorem{prop}[thm]{Proposition}
\newtheorem{cor}[thm]{Corollary}
\newtheorem{lem}[thm]{Lemma}

\theoremstyle{definition}
\newtheorem{dfn}[thm]{Definition}
\newtheorem{defi}[thm]{Definition}

\newtheorem{conj}[thm]{Conjecture}
\newtheorem{conv}[thm]{Notation}
\theoremstyle{remark}
\newtheorem{rmk}[thm]{Remark}
\newtheorem{rem}[thm]{Remark}
\newtheorem{prob}[thm]{Problem}

\newcommand*{\chom}{\mathcal{H}\kern -.5pt om}

\newcommand{\bZ}{\mathbb{Z}}
\newcommand{\bQ}{\mathbb{Q}}
\newcommand{\bR}{\mathbb{R}}
\newcommand{\bC}{\mathbb{C}}

\newcommand{\bS}{\mathbb{S}}
\newcommand{\bT}{\mathbb{T}}

\newcommand{\bP}{\mathbb{P}}

\newcommand{\cA}{\mathcal{A}}

\newcommand{\cC}{\mathcal{C}}

\newcommand{\cF}{\mathcal{F}}

\newcommand{\cH}{\mathcal{H}}

\newcommand{\cL}{\Omega}

\newcommand{\cQ}{\mathcal{Q}}

\newcommand{\cT}{\mathcal{T}}

\newcommand{\X}{\mathcal{X}}%--------------shortened
\newcommand{\cX}{\mathcal{X}}

\newcommand{\Z}{\mathcal{Z}}%--------------shortened

\newcommand{\sA}{\mathsf{A}}
\newcommand{\sB}{\mathsf{B}}

\newcommand{\sD}{\mathsf{D}}

\newcommand{\sF}{\mathsf{F}}

\newcommand{\sP}{\mathsf{P}}
\newcommand{\sQ}{\mathsf{Q}}

\newcommand{\sS}{\mathsf{S}}
\newcommand{\sT}{\mathsf{T}}
\newcommand{\sU}{\mathsf{U}}
\newcommand{\sV}{\mathsf{V}}

\newcommand{\fF}{\mathfrak{F}}
\newcommand{\fS}{\mathfrak{S}}

\newcommand{\rH}{\mathrm{H}}

\newcommand{\Hom}{\mathrm{Hom}}

\newcommand{\tri}{\triangle}
\newcommand{\sgn}{\mathrm{sgn}}
\newcommand{\trop}{\mathrm{trop}}
\newcommand{\stab}{\mathrm{stab}}
\newcommand{\spl}{\mathrm{split}}

\newcommand{\MF}{\mathcal{MF}}

\newcommand{\tr}{\mathsf{T}}
\newcommand{\bep}{\boldsymbol{\epsilon}}

\newcommand{\per}{\mathsf{per}}
\newcommand{\Dfd}{\sD_{\mathsf{fd}}}

\newcommand{\Stab}{\mathrm{Stab}}
\newcommand{\Aut}{\mathrm{Aut}}
\newcommand{\add}{\mathsf{add}}

\newcommand{\Teich}{Teichm\"uller}

\DeclareMathOperator{\interior}{\mathrm{int}}

\newcommand{\bs}{{\boldsymbol{s}}}
\newcommand{\indedge}{t \overbar{\indk} t'}
\newcommand{\edge}{t \overbar{k} t'}
\newcommand{\indi}{i}
\newcommand{\indj}{j}
\newcommand{\indk}{k}
\newcommand{\numi}{i}

\makeatletter
\newcommand{\oset}[3][0ex]{%
  \mathrel{\mathop{#3}\limits^{
    \vbox to#1{\kern-2\ex@
    \hbox{$\scriptstyle#2$}\vss}}}}
\makeatother
\newcommand{\overbar}[1]{\oset{#1}{-\!\!\!-\!\!\!-}}

\makeatletter
\newcommand{\osetnear}[3][0ex]{%
  \mathrel{\mathop{#3}\limits^{
    \vbox to#1{\kern-.3\ex@
    \hbox{$\scriptstyle#2$}\vss}}}}
\makeatother

\title{Categorical dynamical systems arising from sign-stable mutation loops}

\author[Shunsuke Kano]{Shunsuke Kano}
\address{Shunsuke Kano, Research Alliance Center for Mathematical Sciences, Tohoku University, 6-3 Aoba, Aramaki, Aoba-ku, Sendai, Miyagi 980-8578 Japan.}
\email{s.kano@tohoku.ac.jp}

\date{\today}

\begin{document}

\maketitle

\begin{abstract}
We give an autoequivalence of the derived category of the Ginzburg dg algebra for a mutation loop satisfying the sign stability introduced in \cite{IK19}.
We compute the categorical entropies of their restrictions to some subcategories and conclude that they are both given by the logarithm of the cluster stretch factor.
Moreover, we discuss the pseudo-Anosovness of them in the sense of \cite{FFHKL} and \cite{DHKK}.
\end{abstract}

\tableofcontents

\section{Introduction}

A \emph{categorical dynamical system} is a pair $(\sD, F)$ of a triangulated (or $A_\infty$-) category and an exact endofunctor $F$ of the category $\sD$.
The study of the categorical dynamical systems is initiated by Dimitrov--Haiden--Katzarkov--Kontsevich in \cite{DHKK}.
One of the backgrounds of them is the connection of the Teichm\"uller theory with the theory of the stability conditions of the triangulated categories \cite{GMN,BS}.
Namely, they bring Thurston's theory of dynamical systems on the surfaces into the study of the triangulated categories.
They define an invariant of the categorical dynamical systems, which is called \emph{categorical entropy}, as an analogy of the topological entropy.
Roughly speaking, it measures the complexity of the categorical dynamical systems.
Unlike the topological entropy, the categorical entropy has a real parameter, which represents how the endofunctor translates inside the category in the direction of the shift (see \cite{FF}).
The derivative of a categorical entropy at zero with respect to this parameter looks like a dynamical dimension for suitable endofunctors.
In fact, if the endofunctor is the Serre functor, the differential coefficient coincides with the Serre dimension.
This paper study the categorical dynamical systems arising from \emph{sign-stable mutation loops}.

\subsection{Sign-stable mutation loops}
The \emph{mutation} of a quiver is the fundamental deformation of it around a vertex, which is a key notion in the theory of \emph{cluster algebras}.
A \emph{mutation loop} is a sequence of mutations such that the initial and the final quivers are the same.
The mutation loops of a quiver $Q$ form a group which is called the \emph{cluster modular group} of $Q$.
The cluster modular group is a symmetry group of several objects related to cluster algebras.
For instance, if a quiver $Q$ is obtained by an ideal triangulation of a punctured surface $\Sigma$, then the cluster modular group of $Q$ is almost equivalent to the mapping class group of $\Sigma$.
Ishibashi and the author introduced the notion of \emph{sign} of the mutation loops, which is a map from the tropicalized cluster variety to the set of signs $\{+, 0, -\}$.
A mutation loop is said to be \emph{sign-stable} if the sign of it stabilizes via the iterated pullback by the mutation loop.

The sign stability is introduced as a cluster algebraic analog of the \emph{pseudo-Anosovness} of the mapping classes.
In fact, a mutation loop of a quiver associated with an ideal triangulation of a punctured surface is ``uniformly'' sign-stable if and only if the corresponding mapping class is pseudo-Anosov \cite[Theorem 1.2]{IK20a}.
Also, like as the pseudo-Anosov mapping classes, a sign-stable mutation loop has a stretch factor (we call it \emph{cluster stretch factor}) and the algebraic entropy of the induced cluster $\cA$- and $\cX$-transformations $\phi^a$ and $\phi^x$, which are the birational maps of the cluster $\cA$- and $\cX$-varieties, are given by the logarithm of it \cite[Theorem 1.1]{IK19}.

\subsection{Categorification of cluster algebras}
For a quiver with potential $(Q,W)$, namely a pair of a quiver $Q$ and a linear combination $W$ of oriented cycles in $Q$, Ginzburg defined the Calabi--Yau dg algebra $\Gamma_{Q,W}$ \cite{Gin}, which is called Ginzburg dg algebra in the sequel.
The derived category $\sD(\Gamma_{Q,W})$ of the dg algebra $\Gamma_{Q,W}$ has many information of the combinatorics of the cluster algebra of $Q$ \cite{Kel1,Kel2}.
This type of categorification of the cluster algebras called ``additive categorification" of cluster algebras.
Keller--Yang categorify a mutation of a quiver with potential as a generalization of a reflection functor \cite{KY}.
More precisely, they define the derived equivalence between the Ginzburg dg algebras of a quiver with potential and the mutated one.
However, the derived equivalences are uniquely defined up to ``signs''.
Thus, to lift a mutation loop to an autoequivalence of the derived category, we have to choose the signs.
Fortunately, their sign perfectly matches our sign of the mutation loops.

A sign-stable mutation loop has a unique sign, which is called \emph{stable sign}.
Using the stable sign, one can lift a sign-stable mutation loop $\phi$ to an autoequivalence $F_\phi$ of the derived category of the Ginzburg dg algebra.
Thus, one obtains a categorical dynamical system $(\sD(\Gamma_{Q,W}), F_\phi)$, which plays a central role in this paper.

\subsection{Categorical entropy of the categorical dynamical systems defined by the sign-stable mutation loops}
The categorical entropy of several categorical dynamical systems are computed besides in \cite{DHKK} (\emph{e.g.}, \cite{KT19,Kik17,KST20,Ou18,Ou20,Fan18,Yo20,Ike}).
Most of the previous researches are based on algebraic geometry while our objects belong to representation theory or Teichm\"uller theory.
More precisely, we will compute the categorical entropy of the restricted categorical dynamical systems $(\sD(\Gamma_{Q,W}, F_\phi))$ to finite-dimensional $(\Dfd(\Gamma_{Q,W}), F_\phi |_{\Dfd})$ or perfect $(\per(\Gamma_{Q,W}), F_\phi |_\per)$ derived categories.

\begin{thm}[{\cref{thm:cat_entropy_Dfd,thm:cat_entropy_per}}]
Let $\phi$ be a mutation loop with a representation path $\gamma: t \to t'$ which is sign-stable (\cref{d:sign stability}) on $\interior \cC^+_{(t)}$ and let us assume that \cref{p:spec_same} is true.
Then the following hold:
\begin{enumerate}
    \item $h_T(F_\phi|_{\Dfd}) = \log \lambda_\phi$
    for any $T \in \bR$.
    \item $h_0(F_\phi|_\per) = \log \lambda_\phi.$
\end{enumerate}
Here, $h_T(-)$ denotes the categorical entropy and $\lambda_\phi$ denotes the cluster stretch factor of $\phi$.
\end{thm}

Combining the main result in \cite{IK19}, we obtain the following:
\begin{cor}
Assume the representation path $\gamma: t \to t'$ of the mutation loop $\phi$ is sign-stable on $\interior \cC^+_{(t)} \cup \interior \cC^-_{(t)}$.
Then the following holds:
\[
h_T(F_\phi|_{\Dfd}) = h_0(F_\phi|_\per) = h_\mathrm{alg}(\phi^a) =  h_\mathrm{alg}(\phi^x) = \log \lambda_\phi.
\]
Here, $h_\mathrm{alg}(-)$ denotes the algebraic entropy.
\end{cor}

\subsection{Pseudo-Anosovness of the autoequivalences defined by the sign-stable mutation loops}
In the theory of the categorical dynamical systems, some kinds of the pseudo-Anosovness of autoequivalences are suggested \cite{DHKK,FFHKL}.
We discuss the pseudo-Anosovness of the functor $F_\phi$ of a sign-stable mutation loop $\phi$ in the sence of \cite{FFHKL} in general setting:

\begin{thm}[{\cref{thm:pA_auto}}]
If a mutation loop $\phi$ is uniformly sign-stable and
has a North dynamics on $\bS\cC^+_{(t)}$ (\cref{def:NS_dyn} (1)) for some $t \in \bT_I$ with an $\cX$-filling (\cref{def:X-fill}) attracting point $p^+_\phi$ whose stretch factor $\lambda_\phi^{(t)}$ is larger than $1$,
% , that is, $x_i^{(t')}(p^+_\phi) \neq 0$ for any $t' \in \bT_I$ and $i \in I$,
then $F_\phi^{(t)}$ is pseudo-Anosov with a stretch factor $\lambda_{\phi}^{(t)}$.
% which is a cluster stretch factor of $\phi$ at ${(t)}$.
\end{thm}

Moreover, we discuss the relation with the pseudo-Anosovness in the sense of \cite{DHKK} in the case of the quiver obtained from an ideal triangulation of a punctured surface $\Sigma$.
% and a pseudo-Anosovness of a mapping class on $\Sigma$.

\subsection*{Organization of the paper}
In Section 2, we review some notation of cluster algebras (following \cite{IK19}) and basic tools in triangulated categories, and we give a seed pattern via the categorification of cluster algebras.
In Section 3, we review the notion of sign stability of mutation loops and we give the autoequivalence $F_\phi$ associated with a sign stable mutation loop $\phi$, which plays a central role in this paper.
In Section 4, we compute the categorical entropy of the functor $F_\phi$ restricted to finite-dimensional derived category and perfect derived category.
In Section 5, we discuss the pseudo-Anosovness of $F_\phi$ in the sense of \cite{FFHKL}.
In Section 6, we compare the results in \cite{IK20a} and observe the issues of pseudo-Anosovness of $F_\phi$ in the sense of \cite{DHKK}.

\subsection*{Acknowledgements}
% This paper is a substantially revised version of part of Chapter 6 of the author’s Ph.D. thesis.
The author is deeply grateful to his supervisor Yuji Terashima.
Also he would like to thank to Tsukasa Ishibashi, Akishi Ikeda, Kohei Kikuta and Hidetoshi Masai for their valuable discussions and comments.

\section{Preliminaries}
In this section, we fix notation and review some notions.

\subsection{Mutation loops}
First, we give a brief review the notion of sign stability of mutation loops introduced in \cite{IK19}.
% In this preliminary section, we treat only horizontal mutation loops for simplicity same as \cite{IK19}.
% One can find more general setting in \cite{IK20a}.

\subsubsection{Seed patterns}\label{subsec:seed_mut}
We fix a finite index set $I = \{ 1,2, \dots, N \}$ and a regular tree $\bT_I$ of valency $|I| = N$, whose edges are labeled by $I$ so that the set of edges incident to a fixed vertex has distinct labels.

To each vertex $t$ of $\bT_I$, we assign the following data:
\begin{itemize}
    \item A lattice $N^{(t)} = \bigoplus_{\indi \in I} \bZ e_\indi^{(t)}$ with a basis $(e_\indi^{(t)})_{\indi \in I}$.
    \item An integral skew-symmetric matrix $B^{(t)} = (b_{\indi \indj}^{(t)})_{\indi, \indj \in I}$, called the \emph{exchange matrix}.
    %\item The dual lattice $M^{(t)} = \bigoplus_{i \in I} \bZ f_i^{(t)}$, where $(f_i^{(t)})_{i \in I}$ is the dual basis of $(e_i^{(t)})_{i \in I}$.
\end{itemize}
We call such a pair $(N^{(t)}, B^{(t)})$ of data a \emph{seed}.
Let $M^{(t)}:=\Hom (N^{(t)},\bZ)$ be the dual lattice of $N^{(t)}$, and let $(f^{(t)}_\indi)_{\indi \in I}$ be the dual basis of $(e^{(t)}_\indi)_{\indi \in I}$.
% We call the matrix $B^{(t)}$ the \emph{exchange matrix}. 
We define a skew-symmetric bilinear form $\{-,-\}: N^{(t)} \times N^{(t)} \to \bZ$ by $\{e_\indi^{(t)},e_\indj^{(t)}\} := b_{\indi\indj}^{(t)}$.
It induces a linear map $p^*: N^{(t)} \to M^{(t)}$, $n \mapsto \{n,-\}$ called the \emph{ensemble map}.
% where $M^{(t)}$ denotes the dual lattice of $N^{(t)}$. 
Each triple $(N^{(t)}, \{-,-\}, (e_\indi^{(t)})_{\indi \in I})$ is called a seed in \cite{FG09}.

We say that the two seeds $(N, B)$ and $(N', B')$ are \emph{isomorphic} if there is a permutation $\sigma \in \fS_I$ such that $\sigma. B = B'$.
Here, $\sigma. A := (a_{\sigma^{-1}(i), \sigma^{-1}(j)})_{i,j \in I}$ for each matrix $A = (a_{ij})_{i,j \in I}$.

\begin{rem}\label{r:matrix convention}
For exchange matrices we use the notation $B$ rather than $\epsilon$, since we want to reserve the latter for signs $\epsilon \in \{+,0,-\}$. Our exchange matrix is related to the one $B^{\mathrm{FZ}}=(b^\mathrm{FZ}_{\indi\indj})_{\indi.\indj \in I}$ used in \cite{FZ-CA4,NZ12} by the transposition $b^{\mathrm{FZ}}_{\indi\indj} = b_{\indj\indi}$.
\end{rem}
%The data of a seed can be encoded in a quiver $Q$, as follows: $Q$ has vertices indexed by $I$ and $|b_{ij}|$ arrows from $i$ to $j$ (resp. $j$ to $i$) if $b_{ij} >0$ (resp. $b_{ij} < 0$). 

We call such an assignment $\bs: t \mapsto (N^{(t)},B^{(t)})$ a \emph{seed pattern} if for each edge $\indedge$ of $\bT_I$ labeled by $\indk \in I$, the exchange matrices $B^{(t)}=(b_{\indi\indj})$ and $B^{(t')}=(b'_{\indi\indj})$ are related by the \emph{matrix mutation}:
\[ 
b'_{\indi\indj} = 
    \begin{cases}
    -b_{\indi\indj} & \mbox{if $\indi=\indk$ or $\indj=\indk$}, \\
    b_{\indi\indj} + [b_{\indi\indk}]_+ [b_{\indk\indj}]_+ - [-b_{\indi\indk}]_+ [-b_{\indk\indj}]_+ & \mbox{otherwise}.
    \end{cases} 
\]
Here $[a]_+:=\max\{a,0\}$ for $a \in \bR$, throughout this paper.
As a relation between the lattices assigned to $t$ and $t'$, we consider two linear isomorphisms $\widetilde{\mu}_{\indk,\epsilon}^*: N^{(t')} \xrightarrow{\sim} N^{(t)}$ which depend on a sign $\epsilon \in \{+,-\}$ and is given by
\[
e'_\indi \mapsto 
\begin{cases}
    -e_\indk & \mbox{if $\indi = \indk$},\\
    e_\indi + [\epsilon b_{\indi\indk}]_+ e_\indk & \mbox{if $\indi \neq \indk$}.
\end{cases}
\]
Here we write $e_\indi:=e_\indi^{(t)}$ and $e'_\indi:=e_\indi^{(t')}$.
It induces a linear isomorphism $\widetilde{\mu}_{\indk,\epsilon}^*: M^{(t')} \xrightarrow{\sim} M^{(t)}$ (denoted by the same symbol) which sends $f'_\indi$ to the dual basis of $\widetilde{\mu}_{\indk,\epsilon}^*(e'_\indi)$.
Explicitly, it is given by
\[
f'_\indi \mapsto
\begin{cases}
    -f_\indk + \sum_{\indj \in I} [-\epsilon b_{\indk\indj}]_+ f_\indj & \mbox{if $\indi = \indk$},\\
    f_\indi & \mbox{if $\indi \neq \indk$}.
\end{cases}
\]
We call each map $\widetilde{\mu}_{\indk,\epsilon}^*$ the \emph{signed seed mutation} at $\indk \in I$.

One can check the following lemma by direct calculation:
\begin{lem}
The signed mutations are compatible with matrix mutations. Namely, for any $k \in I$ and $\epsilon \in \{+,-\}$, we have
\[
\{\widetilde{\mu}_{\indk,\epsilon}^*(e'_\indi),\widetilde{\mu}_{\indk,\epsilon}^*(e'_\indj)\} = b'_{\indi\indj}.
\]
\end{lem}

For later discussions, we collect here some properties of the presentation matrices of the signed seed mutation and its dual, with respect to the seed bases. 
For an edge $\indedge$ of $\bT_{I}$ and a sign $\epsilon \in \{+,-\}$, let us consider the matrices $\check{E}^{(t)}_{\indk,\epsilon} = (\check{E}_{\indi\indj})_{\indi,\indj \in I}$ and $E^{(t)}_{\indk,\epsilon} = (E_{\indi\indj})_{\indi,\indj \in I}$, given as follows:
\begin{align}
    \check{E}_{\indi\indj}:=
    \begin{cases}
        1 & \mbox{if $\indi=\indj \neq \indk$}, \\
        -1 & \mbox{if $\indi=\indj=\indk$}, \\
        [-\epsilon b_{\indk\indj}^{(t)}]_+ & \mbox{if $\indi=\indk$ and $\indj \neq \indk$}, \\
        0 & \mbox{otherwise},
    \end{cases}\label{eq:E_check_matrix}
\end{align}

\begin{align}
    E_{\indi\indj}:=
    \begin{cases}
        1 & \mbox{if $\indi=\indj \neq \indk$}, \\
        -1 & \mbox{if $\indi=\indj=\indk$}, \\
        [\epsilon b_{\indi\indk}^{(t)}]_+ & \mbox{if $\indj = \indk$ and $\indi \neq \indk$}, \\
        0 & \mbox{otherwise}.
    \end{cases}\label{eq:E_matrix}
\end{align}
Then the transpose of the matrix $E^{(t)}_{\indk,\epsilon}$ gives the presentation matrix of $\widetilde{\mu}_{\indk,\epsilon}^*: N^{(t')} \xrightarrow{\sim} N^{(t)}$ with repsect to the seed bases $(e_\indi^{(t')})$ and $(e_\indi^{(t)})$: $\widetilde{\mu}_{\indk,\epsilon}^*e^{(t')}_\indi= \sum_{\indj \in I}(E^{(t)}_{\indk,\epsilon})_{\indi\indj}e^{(t)}_\indj$.
Similarly the transpose of the matrix $\check{E}^{(t)}_{\indk,\epsilon}$ gives the presentation matrix of $\widetilde{\mu}_{\indk,\epsilon}^*: M^{(t')} \xrightarrow{\sim} M^{(t)}$ with respect to the bases $(f_\indi^{(t')})$ and $(f_\indi^{(t)})$.

The following are basic properties, which can be checked by a direct computation.
\begin{lem}\label{lem:EF_formulae}
For any edge $\edge$ and $\epsilon \in \{+,-\}$, we have the following equations:
\begin{enumerate}
    \item $(E^{(t)}_{\indk, \epsilon})^{-1} = E^{(t)}_{\indk, \epsilon}$, $(\check{E}^{(t)}_{\indk, \epsilon})^{-1} = \check{E}^{(t)}_{\indk, \epsilon}$.
    % \item $(E^{(t)}_{\indk, \epsilon})^{-1} = E^{(t')}_{\indk, -\epsilon}$,
    % $(\check{E}^{(t)}_{\indk, \epsilon})^{-1} = \check{E}^{(t')}_{\indk, -\epsilon}$.
    \item $(E^{(t)}_{\indk, \epsilon})^\tr = (\check{E}^{(t)}_{\indk, \epsilon})^{-1}$.
    % \item $B^{(t)} \check{E}^{(t)}_{\indk, \epsilon} = {E}^{(t)}_{\indk, \epsilon} B^{(t')}$.
\end{enumerate}
\end{lem}
% Here are some comments:
% \begin{itemize}
%     \item By (1), we have $(E^{(t)}_{\indk, \epsilon})^{-1} = E^{(t)}_{\indk, \epsilon}$, $(\check{E}^{(t)}_{\indk, \epsilon})^{-1} = \check{E}^{(t)}_{\indk, \epsilon}$.
%     \item (2) means that the reverse mutation along the same edge has the opposite sign.
%     \item (3) plays an important role in the tropical duality of $C$- and $G$-matrices \cite{NZ12}.
% \end{itemize}

\begin{conv}\label{conv:check}
In the sequel, we use the notation $\check{A}:=(A^\tr)^{-1}$ for an invertible matrix $A$.
% \footnote{Note that this is consistent with the notation $\check{E}^{(t)}_{\indk, \epsilon}$, thanks to \cref{lem:EF_formulae} (3). When one considers a skew-symmetrizable exchange matrix, it should be replaced with $\check{A}:=D(A^\tr)^{-1}D^{-1}$ with a positive integral diagonal matrix $D=\mathrm{diag}(d_1,\dots,d_N)$. }
\end{conv}

\subsubsection{Tropical cluster transformations and tropical cluster varieties}
Let $\mathbb{P}= (\bP, \oplus, \cdot) = \bZ^\trop$ or $\bR^\trop$.
Here $\bZ^\trop$ and $\bR^\trop$ are the semifields $(\bZ, \min, +)$ and $(\bR, \min, +)$ respectively, called \emph{min-plus algebra}.
For a seed pattern $\bs: \bT_I \ni t \mapsto (N^{(t)}, B^{(t)})$, we define the tropical seed $\cX$-tori
% and $\sA$-tori
associated with $t \in \bT_I$ as
\begin{align*}
    \cX_{(t)}(\bP) := M^{(t)} \otimes_\bZ \bP^\times,
    % \ \mbox{ and }\  \sA_{(t)}(\bP) := N^{(t)} \otimes_\bZ \bP^\times
\end{align*}
% respectively,
where $\bP^\times:=(\bP,\cdot)$ denotes the multiplicative group, namely $(\bZ^\trop)^\times = \bZ$ and $(\bR^\trop)^\times = \bR$.
We call the \emph{tropical cluster $\cX$-coordinates}
% and \emph{$\sA$-coordinates}
the functions
\begin{align*}
    x_i^{(t)}&: \cX_{(t)}(\bP) \to \bP, \quad m \otimes p \mapsto \langle e_i^{(t)}, m \rangle p.
    % , \mbox{ and}\\
    % a_i^{(t)}&: \sA_{(t)}(\bP) \to \bP, \quad n \otimes p \mapsto \langle f_i^{(t)}, n \rangle p
\end{align*}
% respectively.
% Here $(e_i^{(t)})_{i \in I}$ denotes the basis of $N^{(t)}$ and $(f_i^{(t)})_{i \in I}$ denotes its dual basis.
% The ensemble map induces a linear map $p_{(t)}: \A_{(t)}(\bP) \to \X_{(t)}(\bP)$.

The tropical cluster $\cX$-transformations $\mu_k^x: \cX^{(t)}(\bR^\trop) \to \cX^{(t')}(\bR^\trop)$ associated with an edge $\indedge$ of $\mathbb{T}_I$ are given by the composition $\widetilde{\mu}_{\indk,\epsilon} \circ \mu_{\indk,\epsilon}^\#$ with the piecewise linear (PL for short) automorphisms $\mu_{\indk,\epsilon}^\#$
given by
\[
(\mu_{\indk,\epsilon}^\#)^*x_\indi := x_\indi -b_{\indi\indk} \min\{0, \epsilon x_\indk\}.
% \quad \mbox{and} \quad (\mu_\indk^{\#,\epsilon})^*A_\indi := A_\indi(1+(p^*X_\indk)^\epsilon)^{-\delta_{\indi\indk}},
\]
% \begin{align}\label{eq:trop x-cluster}
%     (\mu_\indk^{x})^*x'_\indi =
% \begin{cases}
%     -x_\indk & \mbox{if $\indi=\indk$}, \\
%     x_\indi-b_{\indi\indk}\min\{0,-\sgn(b_{\indi\indk})x_\indk\} & \mbox{if $\indi \neq \indk$}.
% \end{cases}
% \end{align}
% and
% \begin{align}\label{eq:trop a-cluster}
%     (\mu_\indk^{a})^*a'_\indi =
% \begin{cases}
%     -a_\indk+ \min \left\{\sum_{\indj \in I}[b_{\indk\indj}]_+a_\indj, \sum_{\indj \in I}[-b_{\indk\indj}]_+a_\indj \right\} & \mbox{if $\indi=\indk$}, \\
%     a_\indi & \mbox{if $\indi \neq \indk$}.
% \end{cases}
% \end{align}
We note that these maps are PL isomorphisms.

\begin{dfn}
The \emph{tropical cluster $\cX$-variety} $\X_{\bs}(\bP)$
% and $\A_{\bs}(\bP)$
associated with a seed pattern $\bs: t \mapsto (N^{(t)},B^{(t)})$ is defined by gluing the corresponding tropical tori by tropical cluster transformations:
\[
\X_{\bs}(\bP) := \bigcup_{t \in \bT_I} \X_{(t)}(\bP).
% , \quad \A_{\bs}(\bP) := \bigcup_{t \in \bT_I} \A_{(t)}(\bP).
\]
Since the tropical cluster transformations are bijective, each $\X_{(t)}(\bP)$ 
% and $\sA_{(t)}(\bP)$
are isomorphic to $\X_{\bs}(\bP)$.
% and $\A_\bs(\bP)$ respectively.
\end{dfn}

\subsubsection{Mutation loops}

A horizontal loop is defined to be an equivalence class of an edge path in $\bT_I$.

\begin{dfn}\label{def:cluster transf}
For an edge path $\gamma:t \to t'$ in $\bT_I$, we define the \emph{tropical cluster transformation} $\mu^x_\gamma: \cX_{(t)}(\bP) \to \cX_{(t')}(\bP)$ associated with $\gamma$ to be the composition of the PL isomorphisms associated with the edges it traverses.
% for $(z,\cZ)=(a,\A),(x,\X)$. 
\end{dfn}
% We note that it only depends on the endpoints of $\gamma$.
% For each $t \in \bT_{I}$, we can identify $\cZ_{(t)}(\bP)$ with $\bigcup_{\sigma \in \fS_I} \cZ_{(t,\sigma)}(\bP)$ via $\psi_\sigma$, where the latter is obtained by patching by isomorphisms $\sigma^z$.
% So we can identify $\cZ_\bs(\bP)$ with $\bigcup_{(t, \sigma) \in \bE_I} \cZ_{(t,\sigma)}(\bP)$.

Let $\bs$ be a seed pattern.
We say that two vertices $t, t' \in \bT_I$ are \emph{$\bs$-equivalent} (and write $t \sim_\bs t'$) if the corresponding seeds are isomorphic, that is,  $\sigma. B^{(t)} = B^{(t')}$ for some $\sigma \in \fS_I$.
Then, the following linear isomorphism gives a seed isomorphism:
\begin{align}\label{eq:seed isom}
    %i_{(t,\sigma),(t',\sigma')}^*:
    (N^{(t')},B^{(t')}) \to (N^{(t)},B^{(t)})\ ;\quad
    e^{(t')}_\numi \mapsto e^{(t)}_{\sigma^{-1}(\numi)}.
\end{align}
Let $\gamma_\nu: t_\nu \to t'_\nu$ be an edge path in $\bT_I$ such that $t_\nu \sim_\bs t'_\nu$ for $\nu=1,2$.
We say that $\gamma_1$ and $\gamma_2$ are \emph{$\bs$-equivalent}
if there exists a path $\delta: t_1 \to t_2$ such that the following diagram commutes:
\begin{equation}\label{eq:equivalence of paths}
\begin{tikzcd}%[column sep=large]
    \cX_{(t_1)}(\bZ^\trop) \ar[r, "\mu^x_{\gamma_1}"] \ar[d, "\mu^x_\delta"'] &\cX_{(t'_1)}(\bZ^\trop) \ar[r, "\sim"]& \cX_{(t_1)}(\bZ^\trop) \ar[d, "\mu_{\delta}^x"]\\
    \cX_{(t_1)}(\bZ^\trop) \ar[r, "\mu^x_{\gamma_2}"] & \cX_{(t'_2)}(\bZ^\trop) \ar[r, "\sim"] & \cX_{((t_2)}(\bZ^\trop).
\end{tikzcd}
\end{equation}
Here the isomorphism $\cX_{(t'_\nu)}(\bZ^\trop) \xrightarrow{\sim} \cX_{(t_\nu)}(\bZ^\trop)$ is induced by the seed isomorphism \eqref{eq:seed isom} for $\nu=1,2$. 
The $\bs$-equivalence class containing an edge path $\gamma$ is denoted by $[\gamma]_{\bs}$.
Note that the commutativity of the diagram \eqref{eq:equivalence of paths} does not depend on the choice of the path $\delta$.

\begin{dfn}[mutation loops]
A \emph{mutation loop} is an $\bs$-equivalence class of an edge path $\gamma:t \to t'$ in the labeled exchange graph $\bT_I$ such that $t \sim_\bs t'$. We call $\gamma$ a \emph{representation path} of the horizontal mutation loop if $\phi=[\gamma]_{\bs}$. 
\end{dfn}

\begin{rmk}
The definition of mutation loops presented above looks very different to the one given in \cite{IK19,IK20a}.
Nevertheless, they are the same definition by \cite[Lemma 3.8]{IK19} and the \emph{periodicity theorem} \cite[Theorem 5.1]{IIKKN}.
\end{rmk}

For a horizontal mutation loop $\phi$, take a representation path $\gamma:t \to t'$.
Then we have the following composite of PL isomorphisms:
\begin{align}\label{eq:coord_expression}
    \phi^x_{(t)}: \X_{(t)}(\bR^\trop) \xrightarrow{\mu_\gamma^x} \X_{(t')}(\bR^\trop) \xrightarrow{\sim} \X_{(t)}(\bR^\trop).
\end{align}
% for $(z,\Z)=(a,\A), (x,\X)$. %Here note that the map $i^z_{t',t}$ is induced by a seed isomorphism.
It induces an automorphism on the tropical cluster variety $\X_\bs(\bR^\trop)$, as follows:
\begin{equation*}
% \begin{tikzcd}
%     \cZ_{(t)} \ar[r, "\mu_\gamma^z"] \ar[d] & \cZ_{(t')} \ar[dl] \ar[r, "i_{t',t}^z"] &\cZ_{(t)} \ar[dl]\\
%     \cZ_\bs \ar[r, dashed] & \cZ_\bs. 
% \end{tikzcd}
\begin{tikzcd}
    \cX_{(t)}(\bR^\trop) \ar[r, "\mu_\gamma^z"] \ar[d] & \cX_{(t')}(\bR^\trop) \ar[d] \ar[r, "\sim"] &\cX_{(t)}(\bR^\trop) \ar[d]\\
    \cX_\bs(\bR^\trop) \ar[r, equal] & \cX_\bs(\bR^\trop) \ar[r, "\phi^x"'] & \cX_\bs(\bR^\trop). 
\end{tikzcd}
\end{equation*}
Here the vertical maps are coordinate embeddings given by the definition of the tropical cluster variety.

% In this paper, we 

\subsection{(Co-)t-structures and tiltings}
% In this section, we recall the definition of Bridgeland's stability conditions of triangulated categories and some properties of it.
For an abelian category $\sA$ and a triangulated category $\sD$, we denote $K_0(\sA)$ and $K_0(\sD)$ by the Gronthendieck group of $\sA$ and $\sD$, respectively.
Also for an additive category $\sB$, $K_0^\spl(\sB)$ denotes the split Grothendieck group of it.
For an object $E$, $[E]$ denotes the class of the (split) Grothendieck group represented by $E$.

For a subcategory $\sU$ of a triangulated category $\sD$, we write $\sU^\bot$ and ${}^\bot \sU$ for the right and left orthogonal subcategory of it respectively:
\begin{align*}
    \sU^\bot &:= \{ V \in \sD \mid \Hom_\sD(U, V) = 0 \mbox{ for all $U \in \sU$} \},\\
    {}^\bot \sU &:= \{ V \in \sD \mid \Hom_\sD(V, U) = 0 \mbox{ for all $U \in \sU$} \}.
\end{align*}
% If $\sU = \operatorname{\mathrm{add}}E$ for $E \in \sD$, we abbreviate $E^\bot = \sU^\bot$ and ${}^\bot E = {}^\bot \sU$.
Moreover, for two subcategories $\sU$ and $\sV$, we write $\sU \star \sV$ for the join of them, namely the full subcategory consists of an object $E \in \sD$ which fits into an exact triangle
\begin{align*}
    U \to E \to V \to U[1]
\end{align*}
for some $U \in \sU$ and $V \in \sV$.

\begin{defi}[t-structure]
A \emph{t-structure} of a triangulated category $\sD$ is a subcategory $\sP$ of $\sD$ such that $\sP[1] \subset \sP$ and $\sD = \sP \star \sP^\bot$.
Moreover, the t-structure is \emph{bounded} if it satisfying $\sD = \bigcup_{i,j} \sP[i] \cap \sP^\bot[j]$.
\end{defi}

For a t-structure $\sP \subset \sD$, its \emph{heart} is the subcategory $\sA = \sP \cap \sP^\bot[1]$.
It is well-known that the heart $\sA$ of a t-structure is abelian category and satisfies $K_0(\sD) \cong K_0(\sA)$.

By the second condition of the definition of t-structure, each object $E \in \sD$ has a decomposition $X \to E \to Y \to X[1]$ with $X \in \sP$ and $Y \in \sP^\bot$.
The correspondences $E \mapsto X$ and $E \mapsto Y$ give functors $t_{\leq 0}: \sD \to \sP$ and $t_{\geq 1}: \sD \to \sP^\bot$ respectively.
Actually the composition $t_{\leq 0} \circ [-1] \circ t_{\geq 1} \circ [1]$ is cohomological.
We write $H_\sA^0$ for it and we write $H^n_\sA := [-n] \circ H^0_\sA \circ [n]$.
For $E \in \sD$, we call $H^n_\sA(E)$ the \emph{$n$-th cohomology of $E$ associated with the heart $\sA$}.
If the t-structure is bounded, the cohomology complex associated with the heart is bounded.

For subcategories of an abelian category, define the orthogonal subcategories and the joins like as above.
A pair of subcategories $(\sT, \sF)$ of an abelian category $\sA$ is a \emph{torsion pair} if $\Hom(\sT, \sF) = 0$ and $\sA = \sT \star \sF$.
Let $(\sT, \sF)$ be a torsion pair in a heart $\sA$ of a t-structure $\sP \subset \sD$.
Then, the subcategory $\sP \star \sT[-1]$ is a new t-structure and its heart is $\sF \star \sT[-1]$.
Also the subcategory ${}^\bot(\sF \star \sP^\bot)$ is a new t-structure and its heart is $\sF[1] \star \sT$.
The former $\sF \star \sT[-1]$ (resp. latter $\sF[1] \star \sT$) is called \emph{right} (resp. \emph{left}) \emph{tilt} of $\sA$ at $(\sT, \sF)$.

Next, we recall the `dual' concept of t-structure in some sense, called a co-t-structure or a weight structure.
\begin{defi}[Co-t-structure]
A \emph{co-t-structure} of a triangulated category $\sD$ is a subcategory $\sQ$ of $\sD$ such that $\sQ[-1] \subset \sQ$ and $\sD = \sQ \star \sQ^\bot$.
Moreover, the co-t-structure is \emph{bounded} if it satisfying $\sD = \bigcup_{i,j} \sQ[i] \cap \sQ^\bot[j]$.
\end{defi}

For a co-t-structure $\sQ \subset \sD$, its \emph{co-heart} is the subcategory $\sB = \sQ \cap \sQ^\bot[-1]$.
In the contrast to the t-structures, the co-heart $\sB$ of a co-t-structure is an  additive category but it is not an abelian category in general.
Nevertheless, we have $K_0^{\mathrm{split}}(\sB) \cong K_0(\sD)$.
% Here, $K_0^{\spl}(\sB)$ denotes the split Grothendieck group of an additive category $\sB$.
Also, the decomposition associated with the second condition of co-t-structure, which is called \emph{weight decomposition}, is `functorial up to morphisms that are zero on cohomology'.
That is, a weight decomposition gives a functor to some quotient category of the homotopy category of complexes over a co-heart (\emph{cf.} \cite[Section 3]{Bon}).
% , thus the cohomology associated with co-heart is ill-defined.
In particular for $E \in \sD$, by choosing the weight decomposition of $E$, we obtain the complex $t^\bullet(E)$ over the co-heart $\sB$ which is called \emph{weight complex} of $E$ (\emph{cf.} \cite[Section 2.2]{Bon}).
The next lemma is useful in later sections.

\begin{lem}[{\cite[Proof of Theorem 3.3.1.I]{Bon}}]\label{lem:split_weight_complex}
Let $D \to E \to F \to D[1]$ be an exact triangle of $\sD$ and take weight decompositions of $D$ and $F$.
Then, there is a weight decomposition of $E$ such that the triangle of weight complexes
\begin{align*}
    t^\bullet(D) \to t^\bullet(E) \to t^\bullet(F) \to t^\bullet(D)[1]
\end{align*}
which splits componentwisely, namely $t^n(E) = t^n(D) \oplus t^n(F)$ for any $n \in \bZ$.
\end{lem}

\subsection{(Co-)stability conditions}

We put
\[\rH := \{ z \in \bC : |z|>0,\ \arg(z)/\pi \in (0,1] \}. \]

\begin{defi}[stability function]
A \emph{stability function} on an abelian category $\sA$ is a group homomorphism $Z: K_0(\sA) \to \bC$ such that $Z(E) \in \rH$ for all $0 \neq E \in \sA$.
\end{defi}

When a stability function $Z: K_0(\sA) \to \bC$ is given, the \emph{phase} $\varphi(E)$ of an object $0 \neq E \in \sD$ is defined to be $\arg(Z(E))/\pi \in (0,1]$.

\begin{defi}[semistable object]
Let $Z$ be a stability function on an abelian category $\sA$.
An object $0 \neq E \in \sA$ is \emph{semistable} (with respect to $Z$) if $\varphi(E') \leq \varphi(E)$ for all subobject $0 \neq E' \subset E$.
\end{defi}

\begin{defi}[Harder--Narasimhan property]
The stability function $Z: K_0(\sA) \to \bC$ is said to have the \emph{Harder--Narasimhan property} if for any $0 \neq E \in \sA$, there is a finite sequence of subobjects of $E$
\[ 0 = E_0 \subset E_1 \subset \cdots \subset E_n = E \]
such that $A_i = E_i/E_{i-1}$ are semistable and whose phases are decreasing:
\[ \varphi(A_1) > \varphi(A_2) > \cdots > \varphi(A_n). \]
\end{defi}

\begin{defi}[stability condition]
A \emph{stability condition} on a triangulated category $\sD$ is a pair $(\sA, Z)$ of a heart $\sA$ of a bounded t-structure on $\sD$ and a stability function $Z$ on $\sA$ with the Harder--Narasimhan property.
\end{defi}

% This definition of the stability condition is equivalent to the definition which giving a central charge and a slicing \cite[Property 5.3]{Bri}.
For a stability condition $\sigma = (\sA, Z)$, one can think $Z$ as a group homomorphism from $K_0(\sD)$ to $\bC$ since $K_0(\sA) \cong K_0(\sD)$.
We call $Z: K_0(\sD) \to \bC$ \emph{central charge} of $\sigma$.
Furthermore, we call an object $E \in \sD$ is \emph{semistable} if $E \in \sA[k]$ for some $k \in \bZ$ and $E[-k]$ is semistable with respect to $Z$.
We call the real number $\varphi_\sigma(E) := \varphi(E[-k]) + k$ the \emph{phase} of $E$ with respect to $\sigma$.
% via $K_0(\sD) \cong K_0(\sA)$, and the subcategories $\cP(\varphi)$ consists of the semistable objects with phase $\varphi$ forms a slicing.

In what follows, we always assume that the Grothendieck groups of any triangulated categories are lattices of finite rank.
Let $\Stab(\sD)$ denotes the set of stability conditions on $\sD$ satisfying \emph{support property}, that is, for some norm $\| - \|$ on $K_0(\sD) \otimes \bR$, there is a constant $C >0$ such that $\| [A] \| < C |Z(A)|$ for each semistable object $A \in \sD$.

In order to explain the metric on the set $\Stab(\sD)$, we prepare the following:
\begin{prop}[Harder--Narasimhan filtration in a triangulated category]\label{prop:HN_filt}
Let $(\sA, Z)$ be a stability condition of a triangulated category $\sD$.
Then, for any object $0 \neq E \in \sD$, there is a sequence of exact triangles
\[\begin{tikzcd}[column sep = small]
0 \ar[r, equal] &[-2mm] E_0 \ar[rr] && E_1 \ar[rr] \ar[ld] && E_2 \ar[ld] \ar[r]& \cdots \ar[r] & E_{n-1} \ar[rr] && E_n \ar[r, equal] \ar[ld] &[-2mm] E \\
&& A_1 \ar[lu, dashed] && \ar[lu, dashed] A_2 &&&& \ar[lu, dashed] A_n
\end{tikzcd}\]
where $A_i$ is semistable of phase $\varphi_i$ for each $i=1, \dots, n$ and $\varphi_1 > \varphi_2 > \cdots > \varphi_n$.
\end{prop}
We call such sequence of exact triangles \emph{Harder--Narasimhan filtration} of $E \in \sD$.
Moreover, $\varphi^+_\sigma(E) := \varphi_1$, $\varphi^-_\sigma(E) := \varphi_n$ and $m_\sigma(E) := \sum_i |Z(A_i)|$ which is called \emph{mass} of $E$.
Since the Harder--Narasimhan filtration is unique up to isomorphism, the above functions are well-defined.

\begin{prop}[{\cite[Proposition 8.1]{Bri}}]
Let $\sD$ be a triangulated category.
The function 
\begin{align*}
    d(\sigma_1, \sigma_2) = \sup_{0 \neq E \in \sD} \left\{
    |\varphi^+_{\sigma_1}(E) - \varphi^+_{\sigma_2}(E)|, |\varphi^-_{\sigma_1}(E) - \varphi^-_{\sigma_2}(E)|, \Big| \log \frac{m_{\sigma_1}(E)}{m_{\sigma_2}(E)} \Big|
    \right\} \in [0, \infty]
\end{align*}
defines a generalized metric on $\Stab(\sD)$.
\end{prop}

This generalized metric induces a topology on $\Stab(\sD)$ and induces a metric space structure on each connected component of it.

\begin{thm}[{\cite{Wo}}]
Led $\sD$ be a triangulated category and let $\Stab^\dagger(\sD)$ be a connected component of $\Stab(\sD)$.
Then metric space $(\Stab^\dagger(\sD), d)$ is complete.
\end{thm}

There are two natural group actions on $\Stab(\sD)$.
The first one is the left action by the group $\Aut(\sD)$ of autoequivalences of $\sD$ defined by
\begin{align*}
    F \cdot \sigma := (F(\sA),\, Z \circ [F^{-1}])
\end{align*}
for $\sigma = (\sA, Z) \in \Stab(\sD)$ and $F \in \Aut(\sD)$.
Here, $[F]$ denotes the induced automorphism on $K_0(\sA)$.
The second one is the right action by the universal cover $\widetilde{GL}\! \, ^+(2, \bR)$ of $GL^+(2, \bR)$ defined by
\begin{align*}
    \sigma \cdot g := \big(\sA',\, \overline{g}^{-1} \circ Z\big)
\end{align*}
for $\sigma = (\sA, Z) \in \Stab(\sD)$ and $g \in \widetilde{GL}\! \, ^+(2, \bR)$.
Here $\sA'$ is the extension closed subcategory of $\sD$ generated by the semistable objects $E \in \sD$ with respect to $\sigma$ such that $\varphi_\sigma(E) \in (\theta(g), \theta(g)+1]$, $\overline{g}$ denotes the projection of $g$ under the covering map and $\theta(g)$ is the angle of the rotation induced by $g$.
We think $\bC$ is a subgroup of $\widetilde{GL}\! \, ^+(2, \bR)$ as $\overline{\lambda} = \exp(\pi \lambda \sqrt{-1})$ and $\theta(\lambda) = \mathrm{Re}(\lambda)$.
In particular, both of the $\Aut(\sD)$-action and the $\bC$-action preserve the metric.

Next, we consider the `dual' concept of stability conditions in some sense, called co-stability conditions.
% \begin{defi}[Co-stability function]
% A co-stability function on an additive category $\sB$ is a group homomorphism $W: K_0^{\spl}(\sB) \to \bC$ such that $Z(E) \in \rH$ for all $0 \neq E \in \sB$.
% \end{defi}
For additive categories, we define the stability function on them in the same way as for abelian categories by using split Grothendieck groups instead of Grothendieck groups.
Also we define the phase of objects in the same way.
We are going to consider the split version of semistable objects and Harder--Narasimhan properties.

\begin{defi}[split semistable object]
Let $W$ be a stability function on an additive category $\sB$.
An object $0 \neq E \in \sB$ is \emph{split semistable} (with respect to $W$) if $\varphi(E') = \varphi(E)$ for any direct summand $E' \subset E$.
\end{defi}

\begin{defi}[split Harder--Narashimhan property]
The stability function $W$ on $\sB$ is said to have the \emph{split Harder--Narasimhan property} if any $0 \neq E_1, E_2 \in \sB$ with $\varphi(E_1) < \varphi(E_2)$ satisfy $\Hom_\sB(E_1, E_2) = 0$, and each object $E \in \sB$ has a decomposition $E = B_1 \oplus B_2 \oplus \cdots \oplus B_s$ such that $B_i$ are split semistable.
\end{defi}

\begin{defi}[co-stability condition]
A \emph{co-stability condition} on a triangulated category $\sD$ is a pair $(\sB, W)$ of a co-heart $\sB$ of a bounded co-t-structure on $\sD$ and a stability function $W$ on $\sB$ with the split Harder--Narasimhan property.
\end{defi}

For a co-stability condition $\tau = (\sB, W)$, we define the split-semistable objects and the phase of objects with respect to $\tau$ in the same way as the stability conditions.

\begin{prop}[split Harder--Narasimhan filtration in a triangulated category]
Let $(\sB, W)$ be a co-stability condition of a triangulated category $\sD$.
Then, for any object $0 \neq E \in \sD$, there is a sequence of exact triangles
\begin{equation}\label{eq:split_HN_filt}
\begin{tikzcd}[column sep = small]
0 \ar[r, equal] &[-2mm] E_0 \ar[rr] && E_1 \ar[rr] \ar[ld] && E_2 \ar[ld] \ar[r]& \cdots \ar[r] & E_{n-1} \ar[rr] && E_n \ar[r, equal] \ar[ld] &[-2mm] E \\
&& B_1 \ar[lu, dashed] && \ar[lu, dashed] B_2 &&&& \ar[lu, dashed] B_n
\end{tikzcd}
\end{equation}
where $B_i$ is split semistable of phase $\varphi_i$ for each $i=1, \dots, n$ and $\varphi_1 < \varphi_2 < \cdots < \varphi_n$.
\end{prop}
We also call such sequence of exact triangles \emph{split Harder--Narasimhan filtration} of $E \in \sD$.
We note that the split Harder--Narasimhan filtration is not unique up to isomorphisms, unlike the stability conditions.
We define the mass $m_\tau(E)$ of $E$ as
\begin{align*}
    m_\tau(E) := \inf \bigg\{\sum_i |W(B_i)|\ \bigg|\ \begin{array}{cc}
         \hspace{-1cm}\mbox{$B_i$ are the factors appearing in the split}\\
         \quad \mbox{Harder--Narasimhan filtration of $E$ as \eqref{eq:split_HN_filt}}
    \end{array} \bigg\}.
\end{align*}

\subsection{Derived category of Ginzburg dg algebras}\label{subsec:Ginzburg_dga}
% First, we define the multiplication $uv$ of paths $u$ and $v$ on a quiver $Q$ as
% \[ uv := \begin{cases}
% \mbox{the concatenation of $u$ and $v$} & \mbox{if } t(u) = s(v), \\
% 0 & \mbox{otherwise},
% \end{cases} \]
% and consider the algebra structure of path algebra $\bC Q$ of $Q$ is given by it.
% Then we have the equivalence between the category of representations of $Q$ over $\bC$ and the category of right $\bC Q$-modules.
% \footnote{
% In \cite{KY}, the multiplication rule is given as the composition of morphisms, so our path algebras are opposite of path algebras in \cite{KY}.}

% \subsection{Ginzburg dg algebras of quivers with potentials.}
In this paper,
% we always assume that any quiver has no loops and no oriented 2-cycles and 
we fix the base field $\bC$.
A quiver with potential is a pair $(Q,W)$ of a finite quiver $Q$
% , we denote by $I$ its set of vertices,
and a potential $W$, which is a finite linear combination of oriented cycles of $Q$.
There is a mutation rule for quiver with potentials \cite{DWZ}.
% as a mutation rule for quivers.
For a quiver with potential $(Q,W)$ and a vertex $k$ of $Q$, $\mu_k(Q,W)$ denotes the mutated quiver with potential at $k$.
The potential $W$ of a quiver $Q$ is \emph{nondegenerate} if any iterated mutation for $(Q,W)$ is compatible with the usual quiver mutation for $Q$.
In our setting, there exists a nondegenerate potential of a quiver \cite[Corollary 7.4]{DWZ}.

For a quiver with potential $(Q,W)$, one can construct a 3-Calabi--Yau algebra $\Gamma_{Q,W}$ called Ginzburg dg algebra (\emph{cf.} \cite[Section 2.6]{KY}).
% We denote by $\Gamma_{Q,W}$ the Ginzburg dg algebra of the opposite quiver with potential $(Q^{\mathrm{op}}, W)$.
We will write $\Dfd(\Gamma_{Q,W})$ for the full subcategory of the derived category $\sD(\Gamma_{Q,W})$ of $\Gamma_{Q,W}$ consists of dg $\Gamma_{Q,W}$-modules whose total cohomologies are finite dimensional and call it \emph{finite dimensional derived category} of $\Gamma_{Q,W}$.
From the definition of 3-Calabi--Yau algebra, the category $\Dfd(\Gamma_{Q,W})$ is 3-Calabi--Yau, namely satisfying
\begin{align*}
    \Hom_{\sD \Gamma_{Q,W}}(M, N) \cong \Hom_{\sD \Gamma_{Q,W}}(N, M[3])^*
\end{align*}
for all $M, N \in \Dfd(\Gamma_{Q,W})$.
% This category $\Dfd(\Gamma_{Q,W})$ plays a central roll in this paper.
We will also consider the another subcategory $\per(\Gamma_{Q,W})$ which is a smallest thick subcategory containing $\Gamma_{Q,W}$, called \emph{perfect derived category} of $\Gamma_{Q,W}$.
Since $\Gamma_{Q,W}$ is homologically smooth, $\Dfd(\Gamma_{Q,W}) \subset \per(\Gamma_{Q,W})$.

The category $\Dfd(\Gamma_{Q,W})$ has a bounded t-structure defined by
\begin{align*}
    \{ M \in \Dfd(\Gamma_{Q,W}) \mid H^p(M) = 0  \mbox{ for all } p>0 \},
\end{align*}
called canonical t-structure.
Its heart $\sA_{Q,W}$, called the canonical heart, is of finite-length, namely it is Artinian and Noetherian as an abelian category.

The category $\per(\Gamma_{Q,W})$ has a co-t-structure defined by the smallest full subcategory containing $\{ \Gamma_{Q,W}[-n] \mid n \in \bZ_{\geq 0} \}$ and closed under extensions and direct summands.
Its co-heart is $\add(\Gamma_{Q,W})$.

Fix a quiver $Q^0$ with the set $I$ of vertices and its nondegenerate potential $W^0$ and fix a vertex $t_0 \in \bT_I$.
We are going to see that the category $\Dfd(\Gamma_{Q^0,W^0})$ gives the categorification of the seed pattern $\bs = \bs_{t_0; Q^0}$ with the exchange matrix at $t_0$ is the matrix $B_{Q^0}$ corresponding to the quiver $Q^0$.

It is known that the category $\Dfd(\Gamma_{Q,W})$ is generated by the simple modules $S_i$, $i \in I$, which are obtained from the simple $Q$-representations concentrated at each vertex $i \in I$.

\begin{lem}[{\cite[Lemma 2.15]{KY}}]\label{lem:dim_Ext}
For each vertices $i,j \in I$, we have
\begin{align*}
    \dim \Hom_{\sD\Gamma_{Q,W}}(S_i, S_j[n]) = 
    \left\{ \begin{array}{ll}
    \delta_{i,j} & \mbox{if }n=0,\\
    Q_{j,i} & \mbox{if }n=1,\\
    Q_{i,j} & \mbox{if }n=2,\\
    \delta_{i,j} & \mbox{if }n=3,\\
    0 & \mbox{otherwise}.
\end{array}\right.
\end{align*}
Here $Q_{i,j} := \# \{ \mbox{$i \to j$ in $Q$} \}$.
\end{lem}

The Grothendieck group $K_0(\Dfd(\Gamma_{Q,W}))$ is free on the basis given by $[S_i]$, $i \in I$.
By definition of $\Dfd(\Gamma_{Q,W})$, the Euler pairing on $K_0(\Dfd(\Gamma_{Q,W}))$
\begin{align*}
    \langle M, N \rangle := \sum_{n \in \bZ} (-1)^n \dim \Hom_{\sD_{Q,W}}(M, N[n])
\end{align*}
is well-defined.
By comparing \cref{lem:dim_Ext}, the matrix of the Euler form with respect to the basis $([S_i])_{i \in I}$ is given by the matrix $B_Q$ corresponding to $Q$.
We attach quiver with potentials $(Q^{(t)}, W^{(t)})$ for $t \in \bT_I$ satisfying $(Q^{(t_0)}, W^{(t_0)}) = (Q^0, W^0)$ and $(Q^{(t')}, W^{(t')}) \sim \mu_k(Q^{(t)}, W^{(t)})$ for each edge $t \overbar{k} t'$ in $\bT_I$.
Here $\mu_k(Q,W)$ denotes a mutation for a quiver with potential $(Q,W)$ and $(Q,W) \sim (Q',W')$ means that the quiver with potentials $(Q,W)$ and $(Q',W')$ are right equivalent (see \cite{DWZ}).
Then, we obtain the seed pattern
\begin{align*}
    \bs^{\mathrm{cat}}: t \mapsto \big(K_0(\Dfd(\Gamma_{Q^{(t)},W^{(t)}})) = {\textstyle \bigoplus_{i \in I}\bZ[S_i^{(t)}]},\ B_{Q^{(t)}} \big).
\end{align*}
Here $S_i^{(t)}$ denotes the simple $\Gamma_{Q^{(t)}, W^{(t)}}$-module for $i \in I$.
It is clear that $\bs^{\mathrm{cat}} = \bs$.

On the other hand, the Grothendieck group $K_0(\per(\Gamma_{Q,W}))$ is free on the basis given by $[P_i]$, $i \in I$, where $P_i$ denotes the projective module $e_i \Gamma_{Q,W}$.
Extends the Euler pairing to $\per(\Gamma_{Q,W})$ in the same way, then we have $\langle [P_i], [S_j] \rangle = \delta_{i, j}$ for any $i, j \in I$.
Thus, one can think that $K_0(\per(\Gamma_{Q,W}))$ is a dual lattice of $K_0(\Dfd(\Gamma_{Q,W}))$ via the Euler pairing.

We will consider the mutation loops of this seed pattern $\bs^\mathrm{cat}$.
In order to lift them to an autoequivalence of the derived category of the Ginzburg dg algebras, we have to restrict them as follows:

\begin{defi}\label{def:QP_mut_loop}
A mutation loop $\phi$ of the seed pattern $\bs^\mathrm{cat}$ is a \emph{QP mutation loop} if there is a representation path $\gamma: t \to t'$ such that the quiver with potentials $(Q^{(t)}, W^{(t)})$ and $(Q^{(t')}, W^{(t')})$ corresponding to $t$ and $t'$ respectively are right equivalent.
\end{defi}

From the nondegeneracy of the potential $W^0$ of $Q^0$, any representation path $\gamma: t \to t'$ of a QP mutation loop $\phi$ satisfies the condition $(Q^{(t)}, W^{(t)}) \sim (Q^{(t')}, W^{(t')})$.

We will write $\Gamma^{(t)} = \Gamma_{Q^{(t), W^{(t)}}}$, $\sD^{(t)} = \sD(\Gamma^{(t)})$, $\per^{(t)} = \per(\Gamma^{(t)})$, $\Dfd^{(t)} = \Dfd(\Gamma^{(t)})$ and $\sA^{(t)} = \sA_{Q^{(t)}, W^{(t)}}$ for simplicity.
The signed seed mutations are categorified as follows:

\begin{thm}[{\cite{KY}}]
Let $\bs^{\mathrm{cat}}$ be the seed pattern as above.
For an edge $t \overbar{k} t'$ of $\bT_I$, there exist two derived equivalences
\begin{align*}
    F^*_{k,+}, F^*_{k,-}: \sD^{(t')} \xrightarrow{\sim} \sD^{(t)}
\end{align*}
such that they restricts to subcategories
\begin{align*}
    F^*_{k,\pm}|_{\Dfd} : \Dfd^{(t')} \xrightarrow{\sim} \Dfd^{(t)}, \quad
    F^*_{k,\pm}|_{\per} : \per^{(t')} \xrightarrow{\sim} \per^{(t)}.
\end{align*}
The equivalences send the standard heart $\sA^{(t')}$ to the right and left tilt of $\sA^{(t)}$ at the torsion pairs $(\sS_k,\, \sS_k^\bot)$ and $({}^\bot \sS_k,\, \sS_k)$ of $\sA^{(t)}$ respectively, where $\sS_k = \sS_k^{(t)} := \operatorname{\mathsf{add}}S_k^{(t)}$.
That is,
\begin{align*}
    F^*_{k,+}(\sA^{(t')}) = \sS_k^\bot \star \sS_k[-1]\ \mbox{ and }\ 
    F^*_{k,-}(\sA^{(t')}) = \sS_k[1] \star {}^\bot \sS_k.
\end{align*}
% Here $S_k = S_k^{(t)}$ is the simple module of $\sD^{(t)}$.
Moreover, they induce signed seed mutations:
\begin{align*}
    [F^*_{k,\epsilon}|_{\Dfd}] &= \widetilde{\mu}_{k,\epsilon}^* : K_0(\sD^{(t')}) \xrightarrow{\sim} K_0(\sD^{(t)}),\\
    [F^*_{k,\epsilon}|_{\per}] &= \widetilde{\mu}_{k,\epsilon}^* : K_0(\per^{(t')}) \xrightarrow{\sim} K_0(\per^{(t)}),
\end{align*}
for $\epsilon \in \{+, -\}$.
Namely, the presentation matrix of $[F^*_{k,\epsilon}|_{\Dfd}]$ (resp. $[F^*_{k,\epsilon}|_{\per}]$) is given by $(E^{(t)}_{k, \epsilon})^\tr$ (resp. $(\check{E}^{(t)}_{k, \epsilon})^\tr$).
\end{thm}
\section{The autoequivalence $F_\phi$ associated to a sign stable mutation loop $\phi$}\label{sec:sign_stab_auto}

\subsection{Sign stability of mutation loops}
% When $V$ and $W$ are equipped with bases, we define the \emph{presentation matrix of $\psi$ at a differentiable point $x$} to be the presentation matrix of the tangent map $(d\psi)_x:T_x V \to T_{\psi(x)} W$ with respect to the given bases.

First of all, we recall the tropical cluster $\X$-transformation $\mu_\indk: \X_{(t)}(\bR^\trop) \to \X_{(t')}(\bR^\trop)$ associated with an edge $t \overbar{k} t'$ in $\bT_I$.

\begin{lem}[{\cite[Lemma 3.1]{IK19}}]\label{l:x-cluster signed}
Fix a point $w \in \X_{(t)}(\bR^\trop)$. 
Then the tropical cluster $\X$-transformation $\mu_\indk: \X_{(t)}(\bR^\trop) \to \X_{(t')}(\bR^\trop)$ is given by
\begin{align}\label{eq:sign x-cluster}
    x^{(t')}_\indi(\mu_\indk(w)) =
\begin{cases}
    -x^{(t)}_\indk(w) & \mbox{if $\indi=\indk$}, \\
    x^{(t)}_\indi(w)+[\sgn(x^{(t)}_\indk(w))b^{(t)}_{\indi\indk}]_+x^{(t)}_\indk(w) & \mbox{if $\indi \neq \indk$}.
\end{cases}
\end{align}
\end{lem}
Here, $\sgn(a)$ denotes the sign of $a \in \bR$:
\[
\sgn(a):=
\begin{cases}
    + & \mbox{ if } a>0,\\
    0 & \mbox{ if } a=0,\\
    - & \mbox{ if } a<0.
\end{cases}
\]
This lemma says that the tropical cluster $\X$-transformation $\mu_\indk: \X_{(t)}(\bR^\trop) \to \X_{(t')}(\bR^\trop)$ is linear on each of the half-spaces
\begin{align*}
    \cH_{\indk,\epsilon}^{x,(t)}:= \{ w \in \X_{(t)}(\bR^\trop) \mid \epsilon x^{(t)}_\indk(w) \geq 0 \}
\end{align*}
for $\indk \in I$, $\epsilon \in \{+,-\}$ and $v \in \bT_I$.

% In order to introduce the \emph{sign} of a path in $\bT_I$, we use the following notation.
% \begin{conv}\label{path convention}
% \begin{enumerate}
    % \item
For an edge path $\gamma: t_0 \overbar{k_0} t_1 \overbar{k_1} \cdots \overbar{k_{h-1}} t_h$ in $\bT_I$ and $i=1,\dots,m$, let $\gamma_{\leq i}: t_0 \xrightarrow{(k_0,\dots,k_{i-1})} t_i$ be the sub-path of $\gamma$ from $t_0$ to $t_i$. Let $\gamma_{\leq 0}$ be the constant path at $t_0$. 
%     % Let $(k_{i(0)},\dots,k_{i(h-1)})$ be the subsequence of $\mathbf{k}=(k_0,\dots,k_{m-1})$ which corresponds to horizontal edges. 
%     \item Fixing the initial vertex $t_0$, we simply denote the tropicalization of the coordinate expression $\phi_{(t_0)}^x: \cX_{(t_0)} \xrightarrow{\mu_\gamma^x} \cX_{(t_m)} \xrightarrow{\sim} \cX_{(t_0)}$ of $\phi$ with respect to $v_0$ by $\phi:=\phi_{(t_0)}: \X_{(t_0)}(\bR^\trop) \to \X_{(t_0)}(\bR^\trop)$, when no confusion can occur.
% \end{enumerate}

% \end{conv}

\begin{dfn}[sign of a path]\label{d:sign}
Let the notation as above, and fix a point $w \in \X_{(t_0)}(\bR^\trop)$.
\item The \emph{sign} of $\gamma$ at $w$ is the sequence $\boldsymbol{\epsilon}_\gamma(w)=(\epsilon_0,\dots,\epsilon_{h-1}) \in \{+,0,-\}^h$ of signs defined by
\[\epsilon_{\nu} := \sgn(x^{(t_{\nu})}_{k_{\nu}}(\mu_{\gamma_{\leq {\nu}}} (w)))\]
for $\nu=0,\dots,h-1$.
% \item Further, the \emph{presentation matrix} $E_\gamma(w)$ of $\gamma$ at $w$ is $E^{\boldsymbol{\epsilon}_\gamma(w)}_\gamma := J_{k_m} \cdots J_{k_1}$,
% where
% \[
% J_{k_i} := \begin{cases}
%     E^{(t_{i(\nu)}, \sigma_{i(\nu)})}_{k_{i(\nu)}, \epsilon_\nu} & \mbox{ if $k_i=k_{i(\nu)}$ is a horizontal edge,}\\
%     P_{k_i} & \mbox{ if $k_i$ is a vertical edge}.
% \end{cases}
% \]
% Here $P_\sigma$ denotes the matrix presentation for $\sigma \in \mathfrak{S}_N$.
\end{dfn}

The next lemma expresses the heart of \cref{d:sign}.

\begin{lem}\label{lem:pres_mat_x}
Let $\boldsymbol{\epsilon}=(\epsilon_0,\dots,\epsilon_{h-1})$ be the sign of a path $\gamma:t \to t'$ at $w \in \cX_{(t_0)}(\bR^\trop)$. If it is strict, namely $\boldsymbol{\epsilon} \in \{+,-\}^h$, then the cluster $\X$-transformation $\mu_\gamma$ is differentiable at $w$, and the presentation matrix is given by $P \cdot E_{\gamma,\bep}$, where $P$ is a presentation matrix of the permutation of the seed isomorphism between $t'$ and $t$, and $E_{\gamma, \bep} :=E_{k_{h-1},\epsilon_{h-1}}^{(t_{h-1})}\dots E_{k_1,\epsilon_1}^{(t_1)} E_{k_0,\epsilon_0}^{(t_0)}$.
\end{lem}

\begin{conj}\label{p:spec_same}
For any point $w \in \X_{(t_0)}(\bR^\trop)$ and a path $\gamma$ which represents a mutation loop, the characteristic polynomials of the matrices $E_{\gamma, \boldsymbol{\epsilon}_\gamma(w)}$ and $\check{E}_{\gamma, \boldsymbol{\epsilon}_\gamma(w)}$ are the same up to an overall sign (\emph{cf.}, \cref{conv:check}).
In particular, the spectral radii of these matrices are the same.
\end{conj}
Note that for an $N \times N$-matrix $E$ with $\det E=\pm 1$, the characteristic polynomials of the matrices $E$ and $\check{E}$ are the same up to an overall sign if and only if the characteristic polynomial $P_E(\nu)$ of $E$ is (anti-)palindromic: $P_E(\nu^{-1}) = \pm \nu^{-N}P_E(\nu)$.

We put the cone
\[
    \cC^+_{(t)}:= \left\{ w \in \X_{(t)}(\bR^\trop) \mid x_i^{(t)}(w) \geq 0 \mbox{ for $i =1,\dots,N$} \right\}
\]
for each $t \in \bT_I$.
By \cite[Lemma 3.8]{IK19}, the sign of a path $\gamma: t \to t'$ at the interior of $\cC^+_{(t)}$ is constant.
We refer to it as the \emph{tropical sign} $\bep^\trop_\gamma$ of $\gamma$.

\begin{dfn}[sign stability]\label{d:sign stability}
Let $\gamma$ be a path as above and suppose that $\phi:=[\gamma]_{\bs}$ is a mutation loop. Let $\Omega \subset \X_{(t_0)}(\bR^\trop)$ be a subset which is invariant under the rescaling action of $\bR_{> 0}$. 
Then we say that $\gamma$ is \emph{sign-stable} on $\Omega$ if there exists a sequence $\boldsymbol{\epsilon}^\stab_{\gamma,\Omega} \in \{+,-\}^h$ of strict signs such that for each $w \in \Omega \setminus \{0\}$, and there exists an nonnegative integer $s \in \bZ_{\geq 0}$ such that   \[\boldsymbol{\epsilon}_\gamma(\phi^n(w)) = \boldsymbol{\epsilon}^\stab_{\gamma, \Omega} \]
for all $n \geq s$. 
We call $\boldsymbol{\epsilon}_{\gamma,\Omega}^\stab$ the \emph{stable sign} of $\gamma$ on $\Omega$.
\end{dfn}

% \begin{rem}[Presentation matrix of a mutation loop]
For a mutation loop $\phi$ and a vertex $v_0 \in \bT_I$, its \emph{presentation matrix} at a differentiable point $w \in  \X_{(t_0)}(\bR^\trop)$ is the presentation matrix $E_\phi^{(t_0)}(w) \in GL_N(\bZ)$ of the tangent map 
\begin{align*}
    (d\phi)_w: T_w \X_{(t_0)}(\bR^\trop) \to T_{\phi(w)} \X_{(t_0)}(\bR^\trop)
\end{align*}
with respect to the basis $(f_i^{(t_0)})_{i \in I}$ of $M^{(t_0)} \otimes_\bZ \bR = \cX_{(t_0)}(\bR^\trop)$.
Note that a choice of a representation path $\gamma: t_0 \to t$ of $\phi$ gives a factorization $E_\phi^{(t_0)}(w)=E_{\gamma,\boldsymbol{\epsilon}_\gamma(w)}$ into elementary matrices $E_{k_{i(\nu)},\epsilon_\nu}^{(t_{i(\nu)})}$ as in \cref{lem:pres_mat_x}.
Thus the sign stability implies that the presentation matrix of $\phi$ at each point $w \in \Omega$ stabilizes:

\begin{cor}\label{cor:asymptotic linearity}
Suppose that $\gamma$ is a path which defines a mutation loop $\phi=[\gamma]_{\bs}$, and it is sign-stable on $\Omega$. Then there exists an integral $N \times N$-matrix $E_{\phi,\Omega}^{(t_0)} \in GL_N(\bZ)$ such that for each $w \in \Omega$, there exists an integer $n_0 \geq 0$ such that $E_\phi^{(t_0)}(\phi^n(w)) = E_{\phi,\Omega}^{(t_0)}$ for all $n \geq n_0$.
\end{cor}
We call $E^{(t_0)}_{\phi,\Omega}$ the \emph{stable presentation matrix} of $\phi$ with respect to the vertex $t_0$. Notice that this matrix only depends on the mutation loop $\phi$ and the vertex $t_0$, and defined when $\phi$ admits a sign-stable representation path starting from $t_0$.

We say that a path $\gamma: t_0 \xrightarrow{\mathbf{k}=(k_0,\dots,k_{h-1})} t$ in $\bT_I$ is \emph{fully-mutating} if
\begin{align*}
    \{k_0,\dots,k_{h-1}\}=I.
\end{align*}

For a path $\gamma: t_0 \to t$ in $\bT_I$ which represents a horizontal mutation loop $\phi$.
For a sign $\bep \in \{+,-\}^h$ we define the cone $\cC^{\bep}_{\gamma}$ as
\[\cC^{\bep}_{\gamma} := \mathrm{rel.cl}\{ w \in \cX_{(t_0)}(\bR^\trop) \mid \bep_\gamma(w) = \bep\} \subset \cX_{(t_0)}(\bR^\trop),\]
where $\mathrm{rel.cl}$ denotes the relative closure.

% and which is sign-stable on $\Omega$, we define the cone $\cC^{\stab}_{\gamma}$ as
% \[\cC^{\stab}_{\gamma, \Omega} := \mathrm{rel.cl}\{ w \in \cX_{(t_0)}(\bR^\trop) \mid \bep_\gamma(\phi^n(w)) = \bep^\stab_{\gamma, \Omega} \mbox{ for any } n \in \bZ_{\geq 0}\}, \]
% where $\mathrm{rel.cl}$ denotes the relative closure.

\begin{lem}\label{lem:str_conv}
If $\gamma$ is fully-mutating, the cone $\cC^{\bep}_{\gamma}$ is rationally polyhedral and strictly convex.
\end{lem}

An $\bR_{\geq 0}$-invariant set $\Omega \subset \cX_{(t_0)}(\bR^\trop) \cong \cX_\bs(\bR^\trop)$ is said to be \emph{tame} if there exists $t \in \bT_I$ such that $\Omega \cap \interior \cC^+_{(t)} \neq \emptyset$.
It is clear that the cone $\cC^+_{(t)}$ is tame.

The following is a direct consequence of \cref{lem:str_conv}.

\begin{thm}[Perron--Frobenius property, {\cite{IK19}}]\label{thm:Perron-Frobenius}
Suppose $\gamma$ is a path as above which represents a mutation loop $\phi=[\gamma]_{\bs}$, and which is sign-stable on a tame subset $\cL$. 
Then the spectral radius of $E^{(t_0)}_{\phi,\cL}$ is attained by a positive eigenvalue $\lambda^{(t_0)}_{\phi,\cL}$, and we have $\lambda^{(t_0)}_{\phi,\cL} \geq 1$.
Moreover if either $\gamma$ is fully-mutating or $\lambda^{(t_0)}_{\phi,\cL} >1$, then one of the corresponding eigenvectors is given by the coordinate vector $\boldsymbol{x}(w_{\phi,\cL})\in \bR^I$ for some $w_{\phi,\cL} \in \cC_\gamma^\stab \setminus \{0\}$.
\end{thm}

In particular $\gamma$ is sign-stable on $\Omega$, we define the stable cone $\cC_{\gamma, \Omega}^\stab$ as
\begin{align*}
    \cC_{\gamma, \Omega}^\stab:= \bigcap_{n \geq 1} \cC_{\gamma^n}^{(\bep_{\gamma, \Omega}^\stab)^n},
\end{align*}
where $\bep^n := (\underbrace{\bep, \dots, \bep}_{n}) \in \{+, -\}^{nh}$ for $\bep \in \{+,-\}^h$.

Here, we give in a sense ``strongest" version of sign stability:
\begin{defi}
A mutation loop $\phi$ is \emph{uniformly sign-stable} if any representation path $\gamma$ of $\phi$ is sign-stable on $\bR_{>0} \cX_\bs(\bZ^\trop)$.
\end{defi}

We write $\bS \Omega := (\Omega \setminus \{0\}) /\bR_{>0}$ for a $\bR_{>0}$-invariant subset $\Omega \subset \cX_\bs(\bR^\trop)$.

\begin{defi}\label{def:NS_dyn}
Let $\phi$ be a mutation loop of a seed pattern $\bs$.
\begin{enumerate}
    \item The mutation loop $\phi$ has a \emph{North dynamics} on a subset $\Pi \subset \bS \cX_\bs(\bR^\trop)$ if there is a point $p^+_\phi \in \bS \cX_\bs(\bR^\trop)$ such that \begin{align*}
    \lim_{n \to \infty} \phi^n(v) = p^+_\phi \in \bS \cX_\bs(\bR^\trop)
    \end{align*}
    for any point $v \in \Sigma$.
    We call the point $p^+_\phi$ the \emph{attracting point} of $\phi$ on $\Pi$ and the value $\lambda_\phi$ so that $\phi(w) = \lambda_\phi w$ for $w \in p^+_\phi$ called \emph{stretch factor} of $\phi$.
    \item The mutation loop $\phi$ has a \emph{North-South dynamics} (on $\bS \cX_\bs(\bR^\trop)$) if there are two fixed points $p^+_\phi, p^-_\phi \in \bS \cX_\bs(\bR^\trop)$ of $\phi$ such that
    \begin{itemize}
        \item if $w \not \in p^-_\phi$, then $\lim_{n\to \infty} [\phi^n(w)] = p^+_\phi$;
        \item if $w \not \in p^+_\phi$, then $\lim_{n\to \infty} [\phi^{-n}(w)] = p^-_\phi$.
    \end{itemize}
\end{enumerate}
Moreover, the North-South dynamics of $\phi$ with fixed points $p^\pm_\phi$ is \emph{non-parabolic} if 
\begin{align}\label{eq:non-para}
    \bep_\gamma(p^+_\phi) \neq \bep_\gamma(p^-_\phi)
\end{align}
for some representation paths $\gamma$ of $\phi$.
\end{defi}

\begin{defi}\label{def:X-fill}
A point $w \in \cX_\bs(\bR^\trop)$ is \emph{$\cX$-filling} if $x^{(t)}_i(w) \neq 0$ for all $i \in I$ and all $t \in \bT_I$.
\end{defi}

\begin{thm}[{\cite[Theorem 5.6]{TS}}]\label{thm:TS_main}
Let $\cC$ be a closed strictly convex and of full dimensional cone in $\bR^N$, $N\geq 2$, and let $E$ be a presentation matrix of a linear map $f: \bR^N \to \bR^N$ such that $f(\cC) \subset \cC$.
Then, the following are equivalent:
\begin{enumerate}
    \item $f^n(\cC) \subset \interior\cC$ for some $n > 0$.
    \item $\ker f \cap \cC = \{0\}$ and $\bigcap_{n \geq 0} f^n(\cC) = \bR_{\geq 0} v$ for some $v \in \interior \cC$.
    \item $\rho(E)$ is the simple eigenvalue such that its eigenvector lies in $\interior\cC$ and it is greater than the modulus of all other eigenvalues. 
\end{enumerate}
\end{thm}

\begin{thm}\label{thm:NS=>unifSS}
If a mutation loop $\phi$ has a non-parabolic North-South dynamics on $\bS \cX_\bs(\bR^\trop)$ with $\cX$-filling fixed points, then it is uniformly sign-stable with the stretch factor larger than $1$.
\end{thm}

\begin{proof}
Let $\gamma$ be a representation path of $\phi$ and let \[\cC^+_\gamma := \bigcap_{n \geq 1} \cC^{\bep^n}_{\gamma^n},\] where $\bep = \bep^+_\gamma := \bep_\gamma(w)$ and $w$ is a point of the attracting point $p^+_\phi \in \bS \cX_\bs(\bR^\trop)$ of $\phi$.
% First we see that the cone $\cC^+_\gamma$ is of full dimensional.
% Let $\fF^+_{(t)}$ be the (Fock--Goncharov) cluster complex, namely it is the simplicial fan defined as
% \[
% \fF^+_{(t)} := \{ \mbox{the faces of $\mu_\gamma^{-1}(\cC^+_{(t')})$} \mid \gamma: t \to t' \mbox{ in } \bT_I \}
% \]
% Since the action of the mutation loop $\phi$ on $\fF^+_{(t)}$ is simplicial, $\phi^n(\cC^+_{(t)})$ is of full dimensional for all $n \in \bZ_{\geq 0}$.
% By the North dynamics of $\phi$ on $\bS \cX_\bs(\bR^\trop)$, $\phi^n(\cC^+_{(t)}) \subset \cC$ for $n \gg 0$
% is simplicial, the cone $\phi^n(\cC^+_{(t)})$ is of maximal dimensional, thus so is $\cC^+_\gamma$.
% By the non-parabolicity, the repelling point $p^-_\phi$ of $\phi$ is not contained in $\cC^+_\gamma$.
% Hence, the North dynamics implies that
% \begin{align}\label{eq:North_dyn}
%     \bigcap_{n \geq 0} \phi^n(\cC^+_\gamma) = p^+_\phi.
% \end{align}
% On the other hand, 
The points in $\Omega_\bs^\bQ$ are sent into $\interior \cC^{\bep_\gamma^+}_\gamma$ by iteration of $\phi$ since $p^-_\phi$ is $\cX$-filling, and $p^+_\phi \subset \interior \cC_\gamma^{\bep_\gamma^+}$ by the $\cX$-fillingness of $p^+_\phi$.
Hence $\phi$ is uniformly sign-stable.
On the other hand, if $\gamma$ satisfies the non-parabolic condition \eqref{eq:non-para}, we have
\begin{align}\label{eq:North_dyn}
    \bigcap_{n \geq 0} \phi^n(\cC^+_\gamma) = p^+_\phi.
\end{align}
Moreover, we have $\phi^n(\cC^+_{(t)}) \subset \cC^+_\gamma$ for $n \gg 0$.
% by the sign stability of $\gamma$ on $\Omega^\bQ_\bs$.
Since the action of the mutation loops on the Fock--Goncharov cluster complex
\footnote{The Fock--Goncharov cluster complex is a simplicial fan on $\cX_\bs(\bR^\trop)$.}
\[
\fF^+_{(t)} := \{ \mbox{the faces of $\mu_\gamma^{-1}(\cC^+_{(t')})$} \mid \gamma: t \to t' \mbox{ in } \bT_I \}
\]
is simplicial, $\phi^n(\cC^+_{(t)})$ is of full dimensional for all $n \in \bZ_{\geq 0}$, thus so is $\cC^+_\gamma$.
Therefore, by \eqref{eq:North_dyn} and \cref{thm:TS_main}, the stretch factor of $\phi$ is lager than $1$.
\end{proof}

At the end of this subsection, we see the relatively easy-to-check sufficient condition of that a mutation loop has a North dynamics.
% Under more stronger sign-stable condition, one can see that the action of $\phi$ on the sphere $\bS \cX_\bs(\bR^\trop) := (\cX_\bs(\bR^\trop) \setminus \{0\}) /\bR_{>0}$ has `North dynamics':
\begin{prop}
Let $\phi$ be a mutation loop with a representation path $\gamma$ which is sign-stable on $\Omega^\bQ_\bs$.
If there is an integer $n \geq 0$ satisfying
\begin{align}\label{eq:cone_stab}
    \cC^\stab_{\gamma} = \cC_{\gamma^n}^{(\bep^\stab_\gamma)^n},
\end{align}
then $\phi$ has a North dynamics on $\bS \Omega^\bQ_\bs$.
% then there is a point $w_{\phi}^\stab \in \cX_\bs(\bR^\trop)$ such that
% \begin{align*}
%     \lim_{n \to \infty} [\phi^n(w)] = [w^+_\phi] \in \bS \cX_\bs(\bR^\trop)
% \end{align*}
% for any $w \in \Omega_\bs^\bQ$.
\end{prop}

\begin{proof}
Let $\gamma$ be such a representation path of $\phi$ and $n \geq 0$ be an integer satisfying \eqref{eq:cone_stab}.
Then, by \cref{lem:str_conv}, the cone $\cC^\stab_{\gamma}$ is rationally polyhedral.
Thus, the generators of $\cC^\stab_{\gamma}$ are contained in $\Omega_\bs^\bQ$.
By definition of sign stability, each generator $w$ of $\cC^\stab_{\gamma}$ satisfies $\phi^m(w) \in \interior\cC^\stab_{\gamma}$ for some $m >0$.
It is equivalent to $\phi^m(\cC_\gamma^\stab) \subset \interior\cC_\gamma^\stab$.
Therefore by the equivalence of (1) and (2) of \cref{thm:TS_main} and \cite[Theorem 5.4]{TS}, $\gamma$ has a North dynamics on $\bS \Omega_\bs^\bQ$.
\end{proof}

\subsection{The autoequivalences associated to the sign-stable mutation loops}
Let us consider the seed pattern $\bs = \bs^{\mathrm{cat}}$ defined in \cref{subsec:Ginzburg_dga}.
For a vertex $t_0 \in \bT_I$, let $\cC^\pm_{(t_0)}:=\{w \in \X_{(t_0)}(\bR^\trop) \mid \pm x_i^{(t_0)}(w) \geq 0,~i \in I\}$.

\begin{defi}\label{def:der_eqv_path}
For a an edge path $\gamma: t_0 \overbar{k_0} t_1 \overbar{k_1} \cdots \overbar{k_{h-1}} t_h$ of $\bT_I$ and a strict sign $\bep = (\epsilon_0, \epsilon_1, \dots, \epsilon_{h-1})$, we define the equivalence $F_{\gamma,{\bep}}^*$ as the following composition:
\begin{align*}
    F_{\gamma,{\bep}}^*: \sD^{(t_h)} \xrightarrow{F^*_{k_{h-1},\epsilon_{h-1}}} \cdots \xrightarrow{F^*_{k_1,\epsilon_1}} \sD^{(t_1)} \xrightarrow{F^*_{k_0,\epsilon_0}} \sD^{(t_0)}.
\end{align*}
% Here $(\epsilon_0, \epsilon_1, \dots, \epsilon_{h-1})$ is the stable sign $\bep^\stab_{\gamma,{\cC^+_{(t_0)}}}$ and $\sD^{(t_0)} \xrightarrow{\sim} \sD^{(t_h)}$ is the equivalence induced by the seed isomorphism.
\end{defi}

Let $\gamma$ be an edge path as \cref{def:der_eqv_path}.

\begin{lem}[{\cite[Theorem 3.5]{Nag}, \cite[Theorem 2.18]{Plab}}]\label{lem:intermediate}
If $\bep$ is the tropical sign $\bep_\gamma^\trop$ of a path $\gamma: t \to t'$, then we have the following:
\begin{enumerate}
    \item The subcategory $F_{\gamma,{\bep}}^*(\sA^{(t')})$ is the right tilt of $\sA^{(t)}$ at some torsion pair $(\sT_\gamma, \sF_\gamma)$.
    \item $F^*_{\gamma, \bep}(\Gamma^{(t')}) \in \add(\Gamma^{(t)}[-1]) \star \add(\Gamma^{(t)})$.
\end{enumerate}
\end{lem}

% \begin{rmk}\label{rmk:intermediate}
% There is a similar statement for the ``minus tropical sign".
% That is, if $\bep$ is the tropical sign of $\gamma$ for the minus seed pattern $-\bs: t \mapsto (N^{(t)}, -B^{(t)})$, we will denote it by $\overline{\bep}^\trop_\gamma$, then
% \begin{enumerate}
%     \item The subcategory $F_{\gamma,{\bep}}^*(\sA^{(t')}[-1])$ is the right tilt of $\sA^{(t)}$ at some torsion pair $(\sT'_\gamma, \sF'_\gamma)$.
%     \item $F^*_{\gamma, \bep}(\Gamma^{(t')}[-1]) \in \add(\Gamma^{(t)}[-1]) \star \add(\Gamma^{(t)})$.
% \end{enumerate}
% This statement is obtained by \cite[Corollary 3.9]{IK19}.
% \end{rmk}

\begin{lem}[{\cite[Proposition 3.1]{Plaa}}]\label{lem:rigid_index}
Suppose that two modules $M_1, M_2 \in \per^{(t)}$ satisfying $M_1, M_2 \in \add(\Gamma^{(t)}[-1]) \star \add(\Gamma^{(t)}) \subset \per^{(t)}$.
Then, $M_1 \cong M_2$ if and only if $[M_1] = [M_2] \in K_0(\per^{(t)})$.
\end{lem}

\begin{prop}\label{prop:indep_of_choice}
For two representation paths $\gamma_1: t \to t_1$ and $\gamma_2: t \to t_2$ of a mutation loop $\phi$, we have the following:
\begin{align}\label{eq:isoms}
    F^*_1(P_i^{(t_1)}) \cong F^*_2(P_i^{(t_2)}),\quad F^*_1(S_i^{(t_1)}) \cong F^*_2(S_i^{(t_2)})
\end{align}
for any $i \in I$.
Here, $F^*_\nu = F^*_{\gamma_\nu, \bep_\nu^\trop}$ for $\nu = 1,2$.
% , there are two $G: \sD^{(t_2)} \to \sD^{(t_1)}$ and $G': \sD^{(t'_2)} \to \sD^{(t'_1)}$ which restrict to perfect derived categories and finite dimensional derived categories and commute the following diagram:
% \begin{equation*}
% \begin{tikzcd}
% \sE^{(t'_1)} \ar[r, "F^*_1|_\sE"] & \sE^{(t_1)}\\
% \sE^{(t'_2)} \ar[r, "F^*_2|_\sE"] \ar[u, "G'|_\sE"] & \sE^{(t_2)}  \ar[u, "G|_\sE"']
% \end{tikzcd}
% \quad,\quad
% \begin{array}{l}
% \sE = \per \mbox{ or } \Dfd,\\
% F^*_\nu = F^*_{\gamma_\nu, \bep^\trop_{\gamma_\nu}} \mbox{ for } \nu=1,2.
% \end{array}
% \end{equation*}
\end{prop}

\begin{proof}
% First, we prove for perfect derived categories.
% It is sufficient to prove that the isomorphisms $F^*_1(G'(P^{(t_2)}_i)) \cong G(F^*_2(P^{(t_2)}_i))$ for $i \in I$ since $\per^{(t)}$ is generated by $P_i^{(t)}$s for each $t \in \bT_I$.
% Let us take a path $\delta: t_1 \xrightarrow{\mathbf{k}} t_2$ and $\delta': t'_1 \xrightarrow{\mathbf{k}} t'_2$, and we put $G:= $
From the definition of the $\bs$-equivalence of paths, we have the equality of the matrices
\begin{align*}
    E_{\gamma_1, \bep^\trop_{\gamma_1}} = E_{\gamma_2, \bep^\trop_{\gamma_2}}.
\end{align*}
It is equivalent to 
\begin{align}\label{eq:K_proj_eq}
    [F^*_1(P_i^{(t_1)})] = [F^*_2(P_i^{(t_2)})] \in K_0(\per^{(t)})
\end{align}
since the presentation matrix of $[F_\nu^*]: K_0(\per^{(t_\nu)}) = M^{(t_\nu)} \to K_0(\per^{(t)}) = M^{(t)}$ is given by $(\check{E}_{\gamma_\nu, \bep^\trop_{\gamma_\nu}})^\tr$ for $\nu = 1,2$.
By \cref{lem:intermediate} (2), we have 
\begin{align*}
    F^*_\nu(P_i^{(t_1)}) \in \add(\Gamma^{(t)}[-1]) \star \add(\Gamma^{(t)})
\end{align*} for $\nu=1,2$ and $i \in I$.
Thus, \cref{lem:rigid_index} tells us that \eqref{eq:K_proj_eq} are equivalent to the first isomorphisms of \eqref{eq:isoms}.
The objects $F^*_1(S_i^{(t_1)})$ are the simples of an abelian category $F^*_1(\sA^{(t_1)})$ since $F^*_1$ is an equivalence.
Moreover, the simple $S \in F^*_1(\sA^{(t_1)})$ is isormorphic to $F^*_1(S^{(t_1)}_i)$ if and only if $\Hom(F^*_1(P^{(t_1)}_i), S) \neq 0$.
Thus the first isomorphisms of \eqref{eq:isoms} imply the second isomorphisms.
\end{proof}

Let $\phi$ be a QP mutation loop with a representation path $\gamma: t \to t'$ which is sign-stable on $\cC^+_{(t)}$.
Then we define the functor $F_\phi = F^{(t)}_\phi$ as
\begin{align*}
    (F^{(t)}_\phi)^{-1}: \sD^{(t)} \xrightarrow{\sim} \sD^{(t')} \xrightarrow{F^*_{\gamma, \bep^\stab}} \sD^{(t)},
\end{align*}
where $\bep^\stab = \bep^\stab_{\gamma, \cC^+_{(t)}}$ and the first equivalence comes from the isomorphism of Ginzburg dg algebras $\Gamma^{(t)} \cong \Gamma^{(t')}$ induced by the right equivalence $(Q^{(t)}, W^{(t)}) \sim (Q^{(t')}, W^{(t')})$.
By \cref{prop:indep_of_choice}, the restrictions of the functors to $\per^{(t)}$ and $\Dfd^{(t)}$ do not depend on the choice of the sign-stable representation paths emanating at $t \in \bT_I$ since
\begin{align*}
    F^*_{\gamma, \bep^\stab} = (F^*_{\gamma^{n}, \bep^\trop_{\gamma^{n}}})^{-1} \circ F^*_{\gamma^{n+1}, \bep^\trop_{\gamma^{n+1}}}
\end{align*}
for $n \in \bZ_{\geq 0}$ such that $\bep_\gamma(\phi^n(w)) = \bep^\stab$.

The statement similar to \cref{lem:intermediate} holds for the functor $F_\phi$:
\begin{lem}\label{lem:F_phi_inter}
Let $t \in \bT_I$.
\begin{enumerate}
    \item The subcategory $F^{(t)}_\phi(\sA^{(t)})$ is the left tilt of $\sA^{(t)}$ at some torsion pair $(\sT^{(t)}_\phi, \sF^{(t)}_\phi)$.
    \item $F^{(t)}_\phi(\Gamma^{(t)}) \in \add(\Gamma^{(t)}) \star \add(\Gamma^{(t)}[1])$.
\end{enumerate}
\end{lem}

\begin{proof}
Let $\gamma: t \to t'$ be a representation path of $\phi$ such that it is sign-stable on $\cC^+_{(t)}$, and let $F := F^*_{\gamma, \bep^\trop_\gamma}$ for short.

\noindent
(1):
We put $\sT' := F^{-1}(\sF_\gamma)$ and $\sF' := F^{-1}(\sT_\gamma[-1])$, where $(\sT_\gamma, \sF_\gamma)$ is the torsion pair of $\sA^{(t)}$ in \cref{lem:intermediate} (1).
Then, $(\sT', \sF')$ is a torsion pair of $\sA^{(t')}$ since $F$ is an exact autoequivalence.
Moreover, we have 
\[F^{-1}(\sA^{(t)}) = F^{-1}(\sT_\gamma \star \sF_\gamma) = \sF'[1] \star \sT'.\]
Hence, the subcategories corresponding to $\sT'$ and $\sF'$ through the equivalence $\Dfd^{(t')} \cong \Dfd^{(t)}$ induced by the equivalence of Ginzburg dg algebras give the desired torsion pair.

\noindent
(2):
Let $[F(\Gamma^{(t')})] = \sum_{i \in I} t_i [P^{(t)}_i]$ for $t_i \in \bZ$ and $I_\pm := \{ i \in I \mid \pm t_i >0 \}$.
Since $F$ is an equivalence, $I = I_+ \sqcup I_-$.
Thus, the minimal presentation of $F(\Gamma^{(t')})$ (\cite[Lemma 3.5]{PLaa}) implies that
\[\add(F(\Gamma^{(t')})) = \add\Big( \bigoplus_{\ell \in I_-} P^{(t)}_\ell[-1]\Big) \star \add\Big( \bigoplus_{j \in I_+}P^{(t)}_j \Big).\]
Therefore, 
\[
\add(F(\Gamma^{(t')})) \star \add(F(\Gamma^{(t')}[1])) \supset  \add\Big( \bigoplus_{j \in I_+}P^{(t)}_j \Big) \star \add\Big( \bigoplus_{\ell \in I_-} P^{(t)}_\ell\Big).
\]
Hence, we have $\Gamma^{(t)} \in \add(F(\Gamma^{(t')})) \star \add(F(\Gamma^{(t')}[1]))$ and it is equivalent to the desired statement.
\end{proof}
\section{Categorical entropy of $F_\phi$}

\subsection{Categorical entropy of exact endofunctors}

\begin{defi}
Let $\sD$ be a triangulated category and $X, Y \in \sD$. The \emph{complexity} of $Y$ relative to $X$ is the function $\delta_-(X,Y) : \bR \to [0, \infty]$ defined by
\[ \delta_T(X,Y) := \begin{cases}
0 & \text{if } Y \cong 0, \\
\inf \left\{ \displaystyle \sum_{k=1}^n e^{i_k T} \left|
\begin{tikzcd}[column sep=tiny, nodes={scale=0.85}]
0 \ar[rr] && Y_1 \ar[rr] \ar[ld] && \cdots \ar[rr] && Y_{n-1} \ar[rr, end anchor={[xshift=-10pt]}] &&\hspace{-10pt} Y \oplus Y' \ar[ld]\\
 & X[i_1] \ar[lu, dashed] &&&&&&X[i_n] \ar[lu, dashed]
\end{tikzcd} \right\} \right. & \text{if } Y \in \langle X \rangle, \\
\infty & \text{if } Y \notin \langle X \rangle,
\end{cases}\]
for $T \in \bR$ and some.
Here, $\langle X \rangle$ denotes the smallest thick subcategory of $\sD$ containing $X$.
\end{defi}

We give some basic properties of the complexity which are used in latter:
\begin{lem}[{\cite[Proposition 2.2]{DHKK}}]\label{lem:complexity}
Let $\sD$ be a triangulated category and $D, E, F \in \sD$. Then
\begin{enumerate}
    \item $\delta_T(D, F) \leq \delta_T(D, E)\delta_T(E, F)$,
    \item $\delta_T(D, E \oplus F) \leq \delta_T(D, E) + \delta_T(D, F)$.
\end{enumerate}
\end{lem}

When a triangulated category $\sD$ is equivalent to $\langle G \rangle$ for some $G \in \sD$, we say that $\sD$ has a \emph{split generator} as $G$.
For instance, $G_{Q,W} := \bigoplus_{i \in I} S_i$ (resp. $\Gamma_{Q,W}$) is a split generator of $\Dfd(\Gamma_{Q,W})$ (resp. $\per(\Gamma_{Q,W})$).

\begin{defi}[categorical entropy]
Let $\sD$ be a triangulated category with a split-generator $G$ and let $F: \sD \to \sD$ be an exact endofunctor. The \emph{entropy} of $F$ is the function $h_-(F) : \bR \to [-\infty, \infty)$ defined by
\[ h_T(F) := \lim_{n \to \infty} \frac{1}{n} \log \delta_T(G, F^n(G)), \] 
where $T \in \bR$.
\end{defi}

It is known that the entropy function is independent of the choice of split generator \cite[Lemma 2.5]{DHKK}.
If $\sD$ has a stability condition $\sigma$, Ikeda defines the mass growth function and he compares it with the entropy function \cite{Ike}.
\begin{defi}[mass growth, \cite{DHKK,Ike}]
Let $\sD$ be a triangulated category with a stability condition $\sigma = (\sA, Z)$ and let $E \in \sD$ be a nonzero object with a Harder--Narasimhan filtration as \cref{prop:HN_filt}.
Then, the \emph{mass of $E$ with a parameter} is the function $m_{\sigma, -}(E): \bR \to \bR_{>0}$ defined by
\begin{align}\label{eq:def_mass_with_param}
    m_{\sigma, T}(E) := \sum_{i=1}^n |Z(A_i)| e^{\varphi_i T}
\end{align}
for $T \in \bR$.
Moreover let $F: \sD \to \sD$ be an endofunctor.
The \emph{mass growth with respect to $\sigma$ and  $F$} is the function $h_{\sigma, -}(F): \bR \to [-\infty, \infty]$ defined by
\begin{align*}
    h_{\sigma, T}(F) := \sup_{E \in \sD} \Big\{ \limsup_{n \to \infty} \frac{1}{n} \log(m_{\sigma, T}(F^n E)) \Big\}
\end{align*}
for $T \in \bR$.
\end{defi}

Here we give some results of \cite{Ike} obtained by using mass growth functions.
\begin{thm}[{\cite[Theorem 1.2]{Ike}}]\label{thm:K_lower_bound}
Let $\sD$ be a triangulated category which has a split generator and a stability condition.
Then we have
\[ \log \rho([F]) \leq h_0(F) \]
for any endofunctor $F: \sD \to \sD$.
Here, $\rho([F])$ denotes the spectral radius of the induced automorphism $[F]: K_0(\sD) \to K_0(\sD)$ on the Grothendieck groups of $\sD$.
\end{thm}

Let $A$ be a dg algebra over $\bC$ such that
\begin{itemize}
    \item $H^k(A) = 0$, for all $k >0$,
    \item $H^0(A)$ is a finite dimensional $\bC$-algebra.
\end{itemize}
We note that Ginzburg dg algebras are satisfying these conditions.

\begin{thm}[{\cite[Proposition 4.3]{Ike}}]\label{thm:cat_entropy_homology}
For an endofunctor $F$ on a finite dimensional derived category $\Dfd(A)$, whose categorical entropy is given by
\[ h_T(F) = \lim_{n \to \infty} \frac{1}{n} \log \left( \sum_{k \in \Z} \dim H^k(F^n G) e^{-kT} \right), \]
where $G$ is a split generator of $\Dfd(A)$.
\end{thm}

\subsection{Computation of the categorical entropy of \texorpdfstring{$F^{(t)}_\phi|_{\Dfd}$}{F|Dfd}}
In what follows, we fix a seed pattern $\bs$ as the seed pattern $\bs^{\mathrm{cat}}$ defined in \cref{subsec:Ginzburg_dga}.

\begin{thm}\label{thm:cat_entropy_Dfd}
Let $\phi$ be a QP mutation loop which has a sign-stable representation path $\gamma: t \to t'$ on $\cC^+_{(t)}$.
Then, we have
\begin{align*}
    h_T(F_\phi^{(t)}|_{\Dfd}) = \log \check{\lambda}^{(t)}_\phi
\end{align*}
for any $T \in \bR$.
Here $\check{\lambda}_\phi^{(t)} = \check{\lambda}^{(t)}_{\phi, \cC^+_{(t)}}$ denotes the spectral radius of $\check{E}^{(t)}_{\phi, \cC^+_{(t)}} = ((E^{(t)}_{\phi, \cC^+_{(t)}})^\tr)^{-1}$ where $E^{(t)}_{\phi, \cC^+_{(t)}}$ is the stable presentation matrix of $\phi$ (defined in \cref{cor:asymptotic linearity}).
\end{thm}

From now on, we sometimes omit the super- or subscript $(t)$ for simplicity.
First, we prove the following:

\begin{lem}\label{lem:cat_entropy_Dfd_0}
$h_0(F_\phi|_{\Dfd}) = \log \check{\lambda}_\phi$.
\end{lem}

\begin{proof}
First, the presentation matrix of $[F_\phi|_{\Dfd}]$ coincides with the matrix $\check{E}_{\phi, \cC^+}$ by the following calculation:
\begin{align*}
    [F_\phi|_{\Dfd}] &= P\, [F^\star_{\gamma, \bep^\stab_\gamma}]^{-1}\\
    &= P\, (E_{k_0,\epsilon_0}^\tr E_{k_1,\epsilon_1}^\tr \cdots E_{k_{h-1},\epsilon_{h-1}}^\tr)^{-1}\\
    &= ((P E_{k_{h-1},\epsilon_{h-1}} \cdots E_{k_1,\epsilon_1} E_{k_{0},\epsilon_{0}})^\tr)^{-1}\\
    &= \check{E}_{\phi, \cC^+}.
\end{align*}
Here, $P$ is a presentation matrix of a perumutation of the seed isomorphims between $t'$ and $t$, and $\bep^\stab_\gamma = (\epsilon_0, \epsilon_1, \dots, \epsilon_{h-1})$.
Hence, we have
\begin{align}\label{eq:lower_bound}
    \log \check{\lambda}_\phi \leq h_0(F_\phi|_{\Dfd})
\end{align}
by \cref{thm:K_lower_bound}.
We now take a representation path $\gamma:t \to t'$ which is sign-stable on $\cC^+_{(t)}$ and let $s$ be an nonnegative integer such that if $n \geq s$ then $\bep_\gamma(\phi^n(w)) = \bep^\stab_{\gamma, \cC^+}$ for any $w \in \interior \cC^+_{(t)}$.
Then, we consider the autoequivalence $F'_\phi$ defined by
\begin{align}\label{eq:F'_phi}
    (F'_\phi)^{-1}: \sD^{(t)} \xrightarrow{\sim} \sD^{(t'')} \xrightarrow{F^*_0} \sD^{(t)},
\end{align}
where $F^*_0 = F^*_{\gamma^s, \bep^\trop_{\gamma^s}}$ and $t''$ is the terminal vertex of the path $\gamma^s$.
Since $G^{(t)} = \bigoplus_i S_i^{(t)}$ is a split generator of $\Dfd^{(t)}$ and $F'_\phi$ is an autoequivalence, $G' := F'_\phi(G^{(t)})$ is also a split generator of $\Dfd^{(t)}$.

We are going to consider the cohomology of $F_\phi G'$ associated with the heart $\sA^{(t)}$.
% First, the presentation matrix of $F^n_\phi F'_\phi|_{\Dfd}$ is $\check{E}_{\gamma^{s+n}, \bep^\trop_{\gamma^{s+n}}}$.
% By the tropical duality \cite[Theorem 1.2]{NZ12}, it coincides with the matrix $\check{E}_{\gamma^{-n-s}, \overline{\bep}^\trop_{\gamma^{-n-s}}}$.
By \cref{lem:F_phi_inter} (1), the cohomology is concentrate to degree $0$ or $-1$.
% Also by \cref{lem:intermediate} (1), the cohomology of $F^n_\phi G'$ (associated to the heart $\sA^{(t)}$) is concentrate to degree 0 or 1.
Hence we have
\begin{align}\label{eq:ent_homo}
    h_T(F_\phi|_{\Dfd}) = \lim_{n \to \infty} \frac{1}{n} \log \big( \dim H^0(F_\phi^n G')  + \dim H^{-1}(F_\phi^n G') e^{T} \big)
\end{align}
by \cref{thm:cat_entropy_homology}.
We put $P_n(T) := \dim H^0(F_\phi^n G')  + \dim H^{-1}(F_\phi^n G') e^{T}$.
In particular, since the objects $F^n_\phi (S'_i)$ again simple, we have
\begin{align*}
    F^n_\phi (S'_i) \in \sA^{(t)} \mbox{ or } \sA^{(t)}[1].
\end{align*}
Here, $S'_i = F'_\phi(S^{(t)}_i)$.
We define $s(n,i) \in \{0, 1\}$ as $F^n_\phi (S'_i) \in \sA^{(t)}[s(n,i)]$.
Then we have 
\begin{align*}
P_n(T) = \sum_{i \in I} \dim H^{-s(n,i)}(S'_i) e^{s(n,i)T}.    
\end{align*}
Moreover we note that $[E] = \sum_n [H^n(E)] \in K_0(\Dfd^{(t)}) \cong K_0(\sA^{(t)})$.
Let $\|\cdot\|$ denotes the $L^1$ norm of $K_0(\Dfd^{(t)})$ or the entrywise $L^1$ norm of matrices.
Then, since $[G^{(t)}] = (1, \dots, 1)^\tr \in K_0(\Dfd^{(t)}) \cong \bZ^I$,
\begin{align*}
   \| [F_\phi^n F'_\phi|_{\Dfd}] \|
   = \sum_{i \in I} \| [F_\phi^n(S'_i)] \|
   &= \sum_{i \in I} \| [H^{-s(n,i)}(F_\phi^n(S'_i))] \|\\
   &= \sum_{i \in I} \dim H^{-s(n,i)}(F_\phi^n(S'_i)) = P_n(0).
\end{align*}
Thus, 
\begin{align*}
    h_0(F_\phi|_{\Dfd})
    &= \lim_{n \to \infty} \frac{1}{n} \log \| [F_\phi^n F'_\phi|_{\Dfd}] \|\\
    &\leq \lim_{n \to \infty} \frac{1}{n}\big( \log\| [F_\phi^n|_{\Dfd}] \| + \log\| [F'_\phi|_{\Dfd}] \| \big)\\
    &= \log \rho([F_\phi|_{\Dfd}]) = \log \check{\lambda}_\phi.
\end{align*}
By combining with \eqref{eq:lower_bound}, we obtain the statement.
\end{proof}

\begin{proof}[Proof of \cref{thm:cat_entropy_Dfd}]
When $T \geq 0$,
\begin{align*}
    P_n(0) \leq P_n(T) \leq P_n(0) e^T,
\end{align*}
and when $T \leq 0$,
\begin{align*}
    P_n(0) e^T \leq P_n(T) \leq P_n(0).
\end{align*}
Therefore, the statement follows from \eqref{eq:ent_homo}.
\end{proof}

\subsection{Computation of the categorcal entropy of \texorpdfstring{$F^{(t)}_\phi|_{\per}$}{F|per}}
First, we give the several properties of the mass growth associated with a co-stability condition.
Let $\tau = (\sB, W)$ be a bounded co-t-structure of a triangulated category $\sD$.
Then one can prove almost every property of the mass function $m_\tau$ in the same way for the stability conditions \cite{Ike} but the next lemma is not obvious:
\begin{lem}
If there is an exact triangle $D \to E \to F \to D[1]$, then
\begin{align*}
    m_\tau(E) \leq m_\tau(D) + m_\tau(F)
\end{align*}
for any $D, E, F \in \sD$.
\end{lem}
\begin{proof}
First, we note that when we take a weight complex $t^\bullet(D)$ of $D$ associated with some weight decomposition, we have the following sequence of exact triangles:
\begin{equation*}
\begin{tikzcd}[column sep = tiny]
0 \ar[r, equal] & D_0 \ar[rr] && D_1 \ar[rr] \ar[ld] && D_2 \ar[ld] \ar[r] &[2mm] \cdots \ar[r] &[2mm] D_{n-1} \ar[rr] && D_n \ar[r, equal] \ar[ld] & D \\
&& t^{k_1}(D) \ar[lu, dashed] && \ar[lu, dashed] t^{k_2}(D) &&&& \ar[lu, dashed] t^{k_n}(D)
\end{tikzcd}
\end{equation*}
where $k_1 < k_2 < \cdots < k_n$ is the degree of nonzero factors of $t^\bullet(D)$.
All split Harder--Narasimhan filtration of $D$ is obtaind by taking split Harder--Narasimhan filtrations of $t^{k_i}(D)$s.
Thus, for any $\varepsilon > 0$ there is a weight decomposition of $D$ such that
\begin{align*}
    \sum_{n \in \bZ} m_\tau(t^n(D)) < m_\tau(D) + \varepsilon.
\end{align*}
Now, we fix $\varepsilon>0$ and take weight deompositions of $D$ and $F$ such that
\begin{align*}
    \sum_{n \in \bZ} m_\tau(t^n(D)) < m_\tau(D) + \varepsilon, \quad
    \sum_{n \in \bZ} m_\tau(t^n(F)) < m_\tau(F) + \varepsilon,
\end{align*}
and then, we take the weight decomposition of $E$ as \cref{lem:split_weight_complex}.
% We now take weight decompositions of $D$ and $F$ and take the weight decomposition of $E$ as \cref{lem:split_weight_complex}.
% Then
Then
\begin{align*}
    m_\tau(E)
    &\leq \sum_{n \in \bZ} m_\tau(t^n(E))\\
    &= \sum_{n \in \bZ} m_\tau(t^n(D) \oplus t^n(F))\\
    &= \sum_{n \in \bZ}\big( m_\tau(t^n(D)) + m_\tau(t^n(F)) \big)
    < m_\tau(D) + m_\tau(F) + 2\varepsilon.
\end{align*}
The arbitrariness of $\varepsilon$ implies the result.
\end{proof}

% Next lemma could be proven in the same way as \cite{Ike}.
% \begin{lem}
% For a co-stability condition $\tau$ of $\sD$ and objects $D,E \in \sD$, we have
% \begin{align*}
%     m_\tau(E) \leq m_\tau(D) \delta_0(D, E).
% \end{align*}
% \end{lem}
We define the mass growth entropy $h_\tau(F)$ of an endofunctor $F: \sD \to \sD$ associated with a co-stability condition $\tau$ in the same way as the stability conditions:
\begin{align*}
    h_\tau(F) := \sup_{E \in \sD} \Big\{ \limsup_{n \to \infty} \frac{1}{n} \log(m_\tau(F^n E)) \Big\}.
\end{align*}

Then, one can prove the following properties in the same way as \cite{Ike}
\begin{prop}\label{prop:co-t_mass_growth}
Let $\tau$ be a co-t-structure of $\sD$.
Then we have
\begin{align*}
    \rho([F]) \leq h_\tau(F) \leq h_0(F).
\end{align*}
\end{prop}

By using them, we prove the following:
\begin{thm}\label{thm:cat_entropy_per}
Let $\phi$, $\gamma$ be as \cref{thm:cat_entropy_Dfd}.
Then, we have
\begin{align*}
    h_T(F_\phi^{(t)}|_{\per}) \leq \log \lambda^{(t)}_\phi.
\end{align*}
for any $T \in \bR$.
Especially, $h_0(F_\phi^{(t)}|_{\per}) = \log \lambda^{(t)}_\phi.$
Here
$\lambda_\phi = \lambda^{(t)}_{\phi, \cC^+_{(t)}}$ denotes the spectral radius of the stable presentation matrix $E^{(t)}_{\phi, \cC^+_{(t)}}$ of $\phi$.
% and $\check{\lambda}^{(t)}_\phi$ denotes the spectral radius of $\check{E}^{(t)}_{\phi, \cC^+_{(t)}}$.
\end{thm}
\begin{proof}
Since the presentation matrix of $[F_\phi|_\per]$ coincides with the presentation matrix $E_{\phi, \cC^+}$, we have
\begin{align*}
    {\lambda}_{\phi} \leq h_0(F_\phi|_\per)
\end{align*}
by \cref{prop:co-t_mass_growth}.
On the other hand, we have
\begin{align*}
    F_\phi(P'_i) \in \add(\Gamma^{(t)}) \star \add(\Gamma^{(t)}[1])
\end{align*}
by \cref{lem:F_phi_inter} (2).
Here $P'_i := F'_\phi(P_i)$ and the autoequivalence $F'_\phi$ is defined in the proof of \cref{lem:cat_entropy_Dfd_0}.
Then, there is an exact triangle
\begin{align}\label{eq:min_pres}
    T_1 \to T_0 \to F_\phi(P'_i) \to T_1[1]
\end{align}
in $\per^{(t)}$ such that $T_0, T_1 \in \add(\Gamma^{(t)})$ and they have no direct summands in common by \cite[Lemma 3.5]{Plaa}.
By taking some direct sums \eqref{eq:min_pres} with the trivial exact triangles of the following forms
\begin{align*}
    0 \to P \to P \to 0 \quad \mbox{or} \quad
    P[-1] \to 0 \to P \to P
\end{align*}
for some $P \in \add(\Gamma^{(t)})$, we obtain
\begin{align}\label{eq:good_pres}
    \Gamma^{\oplus \| [T_1] \|} \to \Gamma^{\oplus \| [T_0] \|} \to F_\phi(P'_i) \oplus X \to \Gamma^{\oplus \| [T_1] \|}[1]
\end{align}
for some $X \in \add(\Gamma)$.
Here $\| \cdot \|$ denotes the $L^1$ norm on $K_0(\per^{(t)})$.
Thus we get
\begin{align*}
    \delta_0(\Gamma, F_\phi^n(P'_i)) \leq \| [F_\phi(P'_i)] \|.
\end{align*}
Let we put $\Gamma' = F'_\phi(\Gamma^{(t)})$.
Then, by \cref{lem:complexity},
\begin{align*}
    \delta_0(\Gamma', F_\phi^n(\Gamma'))
    &\leq \delta_0(\Gamma', \Gamma)\delta_0(\Gamma, F_\phi^n(\Gamma'))\\
    &\leq \delta_0(\Gamma', \Gamma) \Big( \sum_{i \in I} \delta_0(\Gamma, F_\phi^n(P'_i)) \Big)\\
    &\leq \delta_0(\Gamma', \Gamma) \| [F_\phi^n F'_\phi] \|\\
    &\leq \delta_0(\Gamma', \Gamma) \| [F_\phi^n] \| \| [F'_\phi] \|
\end{align*}
Since $\Gamma'$ is a split generator of $\per^{(t)}$, this inequality implies
\begin{align*}
    h_0(F_\phi|_\per) &\leq \log \lambda_\phi.
\end{align*}
Also, we have
\begin{align*}
    \delta_T(\Gamma', F_\phi^n(\Gamma')) \leq
    \begin{cases}
    \delta_T(\Gamma', \Gamma) \| [F_\phi^n] \| \| [F'_\phi] \| e^T & \mbox{if } T \geq 0,\\
    \delta_T(\Gamma', \Gamma) \| [F_\phi^n] \| \| [F'_\phi] \| & \mbox{if } T \leq 0.\\
    \end{cases}
\end{align*}
By taking a limit $n \to \infty$, we obtain the desired inequality.
\end{proof}

\begin{rmk}
\begin{enumerate}
\item The reason why we did not use the growth of extensions \cite[Theorem 2.6]{DHKK} is the Ginzburg dg algebra is not compact in general.
\item The perfect derived category $\per(\Gamma)$ of the Ginzburg dg algebra $\Gamma$ has a t-structure induced by a canonical t-structure of the whole derived category $\sD(\Gamma)$ \cite{Ami}.
However, this t-structure is not bounded.
% so there are not stability conditions whose hearts are the heart of that t-structure.
\item Let us define the mass $m_{\tau, T}(E)$ with a parameter $T \in \bR$ with respect to a co-stability condition $\tau$
% \begin{align*}
%     m_{\tau, T}(E) := \inf \bigg\{\sum_i |W(B_i)| e^{\varphi_i T} \ \bigg|\ \begin{array}{cc}
%          \hspace{-1cm}\mbox{$B_i$ are the factors appearing in the split}\\
%          \quad \mbox{Harder--Narasimhan filtration of $E$ as \eqref{eq:split_HN_filt}}
%     \end{array} \bigg\}.
% \end{align*}
by mimicking \eqref{eq:def_mass_with_param}.
% For $T_1 < 0 < T_2 \bR$, we suppose that the masses $m_{\tau, T_1}(E)$ and $m_{\tau, T_2}(E)$ are attained by some split Harder--Narasimhan filtrations of $E$ respectively.
% Then, these filtrations are distinct in general.
% Therefore some estimations corresponding to them considered in \cite{Ike} are more difficult.
Recall that the split Harder--Narasimhan filtration is not unique, so we take an infimum to define the mass $m_\tau(E)$, and thus so is $m_{\tau, T}(E)$.
It makes difficult to estimate the entropy $h_{\tau, T}(F)$ from below for a general $T \in \bR$.
\end{enumerate}
\end{rmk}

\section{Pseudo-Anosovness of $F_\phi$}
\subsection{Pseudo-Anosov autoequivalence}
Let $\sD$ be a triangulated category and assume $\Stab(\sD) \neq 0$.
First, we recall some fundamental properties of the mass growths.
\begin{defi}[{\cite{FFHKL}}]
Let $\sigma$ be a stability condition of $\sD$, $F: \sD \to \sD$ be an exact endofunctor and $E \in \sD$.
The \emph{mass growth of $E$ with respect to $\sigma$ and $F$} is defined to be
\begin{align*}
    h_{\sigma, T}(F, E) := \limsup_{n\to \infty} \frac{1}{n} \log (m_{\sigma, T}(F^n E)),
\end{align*}
for $T \in \bR$.
\end{defi}

\begin{prop}[{\cite[Proposition 3.10, Theorem 3.14]{Ike}}]\label{prop:mass_properties}
The notation as above.
\begin{enumerate}
    \item If $\sigma'$ is a stability condition contained in the same connected component with $\sigma$, then
    \begin{align*}
        h_{\sigma, T}(F, E) = h_{\sigma', T}(F, E).
    \end{align*}
    \item If $\sD$ has a split generator $G$ and the heart of $\sigma$ is algebraic, then
    \begin{align*}
        h_{\sigma, T}(F) = h_{\sigma, T}(F, G).
    \end{align*}
\end{enumerate}
\end{prop}

In what follows, we fix a connected component $\Stab^\dagger(\sD)$ of $\Stab(\sD)$.
Then, thanks to \cref{prop:mass_properties} (1), we do not have to express the stability condition $\sigma$ in $h_{\sigma, T}$ explicitly.
So we will refer to $h_{\sigma, 0}(F, E)$ for $\sigma \in \Stab^\dagger(\sD)$ as $h^\dagger_F(E)$.
When $\sD$ has a split generator $G$, \cref{prop:mass_properties} (2) tells us that $h^\dagger_F(E) \leq h^\dagger_F(G)$ for any nonzero object $E \in \sD$.

\begin{defi}[pseudo-Anosov autoequivalence, \cite{FFHKL}]
An autoequivalence $F$ of a triangulated category $\sD$ is \emph{pseduo-Anosov} if there is a constant $\lambda_F >1$, called \emph{stretch factor} of $F$, such that
\begin{align*}
    h^\dagger_F(E) = \log \lambda_F
\end{align*}
for any nonzero object $E \in \sD$.
\end{defi}

The following is the main theorem of this section.

\begin{thm}\label{thm:pA_auto}
If a QP mutation loop $\phi$ is uniformly sign-stable and
has a North dynamics on $\bS\cC^+_{(t)}$ for some $t \in \bT_I$ with an $\cX$-filling attracting point $p^+_\phi$ whose stretch factor $\lambda_\phi^{(t)}$ is larger than $1$,
% , that is, $x_i^{(t')}(p^+_\phi) \neq 0$ for any $t' \in \bT_I$ and $i \in I$,
then $F_\phi^{(t)}$ is pseudo-Anosov with a stretch factor $\lambda_{\phi}^{(t)}$.
% which is a cluster stretch factor of $\phi$ at ${(t)}$.
\end{thm}

\begin{prop}\label{prop:N_dyn_stab}
If $\phi$ satisfies the same assumption as \cref{thm:pA_auto}.
Then, there is a stability condition $\sigma_\phi^{(t)} \in \Stab(\Dfd^{(t)})$ such that 
\begin{align*}
    F_\phi^{(t)} \cdot \sigma_\phi^{(t)} = \sigma_\phi^{(t)} \cdot \frac{\sqrt{-1}}{\pi}\log \lambda^{(t)}_\phi.
\end{align*}
% Here, $\lambda_\phi^{(t)}$ is the cluster stretch factor of $\phi$.
\end{prop}

\begin{proof}
Let $\sigma_0 = (\cA^{(t)}, Z_0)$ be the stability condition on $\Dfd^{(t)}$ so that $Z_0(S_i) := \sqrt{-1}$.
% and let $\sigma' := F'_\phi \cdot \sigma_0$.
Then, we define the sequence of stability conditions $(\sigma_n)_{n \geq 1}$ as
\begin{align*}
    \sigma_n := F_\phi^n F'_\phi \cdot \sigma_0 \cdot \frac{\log \| \phi^{s+n}(\ell_{(t)}^+) \|}{\pi \sqrt{-1}}.
\end{align*}
Here $F'_\phi$ is an automorphism of $\Dfd^{(t)}$ defined as \eqref{eq:F'_phi}, $\ell_{(t)}^+ \in \cX_{(t)}(\bR^\trop)$ is a point defined by $x^{(t)}_i(\ell_{(t)}^+)=1$ for all $i \in I$ and $\|\cdot\|$ is the Euclidean norm of $M^{(t)} \otimes \bR \cong \cX_{(t)}(\bR^\trop)$.
We note that there is a map
\begin{align*}
    \Pi: \Stab(\Dfd^{(t)}) \to M^{(t)} \otimes \bR \cong \cX_{(t)}(\bR^\trop)\ ;\quad
    (\sA, Z) \mapsto \mathrm{Im}(Z).
\end{align*}
Through this map, the sequence $(\sigma_n)_n$ becomes the following sequence of normalized points in $\cX_{(t)}(\bR^\trop)$:
% That is, one can verify that
\begin{align*}
    \Pi(\sigma_n) =
    \begin{cases}
    \ell_{(t)}^+ & \mbox{ if } n=0, \\
    \phi^{s+n}(\ell_{(t)}^+)/ \| \phi^{s+n}(\ell_{(t)}^+)\| & \mbox{ if } n>0.
    \end{cases}
\end{align*}
We are going to estimate $d(\sigma_m, \sigma_n)$ for $m, n \gg 0$.
Since for simples $S_i = S_i^{(t)}$, $(F_\phi^n F'_\phi)^{-1} (S_i)$ are again simples of $(F_\phi^n F'_\phi)^{-1}(\sA^{(t)})$,
% have the same semistable objects and Harder--Narasimhan filtrations, 
\begin{align*}
    d(\sigma_m, \sigma_n) = \max_{i \in I} \Bigg\{ |\varphi^+_{\sigma_m}(S_i) - \varphi^+_{\sigma_n}(S_i)|, |\varphi^-_{\sigma_m}(S_i) - \varphi^-_{\sigma_n}(S_i)|, \Big| \log \frac{|Z_{\sigma_m}(S_i)|}{|Z_{\sigma_n}(S_i)|} \Big| \Bigg\}.
\end{align*}
% Here  are the simple objects of $\sA^{(t)}$.
By \cref{lem:intermediate} (1), $|\varphi^\pm_{\sigma_m}(S_i) - \varphi^\pm_{\sigma_n}(S_i)| = 0, 1$.
Through the map $\Pi$, we can see that $|\varphi^\pm_{\sigma_m}(S_i) - \varphi^\pm_{\sigma_n}(S_i)| = 1$ if and only if $x^{(t)}_i(\phi^{s+m}(\ell^+_{(t)}))$ and $x^{(t)}_i(\phi^{s+n}(\ell^+_{(t)}))$ have the different signs.
% Since $Z_{\phi_n}(S_i) \in \sqrt{-1} \bR$, $|\varphi^\pm_{\sigma_m}(S_i) - \varphi^\pm_{\sigma_n}(S_i)| = 1$ if and only if $Z_{\sigma_m}(S_i)/\sqrt{-1}$ and $Z_{\sigma_n}(S_i)/\sqrt{-1}$ have different signs.
% For a representation path $\gamma$ of $\phi$, $\tilde{\gamma}$ denotes the representation path of $\phi$ obtained by inserting the round trips $t_r \overbar{i} t'_r \overbar{i} t_r$ at each vertex $t_r$ through which $\gamma$ passes.
% If $x^{(t)}_i(\phi^{s+m}(\ell^+_{(t)}))$ and $x^{(t)}_i(\phi^{s+n}(\ell^+_{(t)}))$ have the different signs, then the path $\tilde{\gamma}$ is not sign-stable.
% It is a contradiction to the hypothesis.
The $\cX$-fillingness of the attracting point $p^+_\phi$ implies that $\sgn(x_i^{(t)}(\phi^m(\ell_{(t)}^+))) = \sgn(x_i^{(t)}(\phi^n(\ell_{(t)}^+)))$ for $m,n \gg 0$ since $p^+_\phi$ is contained in the open subset \[\bS \{ w \in \cX_\bs(\bR^\trop) \mid \sgn(x_i(w)) = \sgn(x_i(p_\phi^+)) \mbox{ for all } i \in I \} \subset \bS \cX_\bs(\bR^\trop).\]
Therefore, $(\sigma_n)_{n \geq 1}$ is Cauchy sequence if so is the sequence $(\Pi(\sigma_n))_{n \geq 1}$.
It follows from \cite[Theorem 5.4]{TS} that the sequence $(\Pi(\sigma_n))_n$ is Cauchy.
\end{proof}

One can prove the following lemma in the same manner as \cite[Theorem 2.15]{FFHKL}.

\begin{lem}\label{lem:Fsigma=sigmalambda}
Let $\sD$ be a triangulated category and $F \in \Aut(\sD)$.
If there is a stability condition $\sigma \in \Stab(\sD)$ and $\lambda \geq 1$ such that
\begin{align*}
    F \cdot \sigma = \sigma \cdot \frac{\sqrt{-1}}{\pi}\log \lambda.
\end{align*}
Then $F$ is pseudo-Anosov.
\end{lem}

By combining \cref{prop:N_dyn_stab} and \cref{lem:Fsigma=sigmalambda}, we obtain \cref{thm:pA_auto}.
Moreover, by \cref{thm:NS=>unifSS}, we obtain the following:

\begin{cor}\label{cor:NS_pA_auto}
If a mutation loop $\phi$ has a non-parabolic North-South dynamics on $\bS \cX_\bs(\bR^\trop)$ with $\cX$-filling fixed points whose stretch factor is $\lambda_\phi >1$, then $F_\phi^{(t)}$ is pseudo-Anosov with a stretch factor $\lambda_{\phi}$.
% which is a cluster stretch factor of $\phi$ at $t \in \bT_I$.
\end{cor}

\begin{rmk}
The assumption of \cref{thm:pA_auto} is very strong, and especially the condition that North dynamics looks unnatural and it looks far from the sign stability.
Although, we conjecture that the direct relationship between North (-South) dynamics and the uniformly sign stability:

\begin{conj}[{\cite[Conjecture 7.9]{IK20a}}]
Let $\bs$ be any seed pattern and $\phi$ is a mutation loop.
Then $\phi$ is uniformly sign-stable with cluster stretch factor greater than $1$ if and only if $\phi$ has North-South dynamics on $\bS \cX_\bs(\bR^\trop)$ with $\cX$-filling attracting/repelling points.
\end{conj}

Under this conjecture, the assumption of \cref{thm:pA_auto} is rewritten shortly as follows:
\begin{conj}
If a mutation loop $\phi$ is uniformly sign-stable with a cluster stretch factor $\lambda^{(t)}_\phi >1$ at $t \in \bT_I$, then the functor $F^{(t)}_\phi$ is pseudo-Anosov with a stretch factor $\lambda^{(t)}_\phi$.
\end{conj}

\end{rmk}

\section{Example: pseudo-Anosov mapping classes on marked surfaces}

\subsection{Corollaries}

Let $\Sigma$ be a marked surface, that is, a compact oriented surface with a fixed non-empty finite set of \emph{marked points} on it. A marked point is called a \emph{puncture} if it lies in the interior of $\Sigma$, and a \emph{special point} otherwise. 
Let $P=P(\Sigma)$ (resp. $M_\partial=M_\partial(\Sigma)$) denote the set of punctures (resp. special points). A marked surface is called a \emph{punctured surface} if it has empty boundary (and hence $M_\partial=\emptyset$). 
We denote by $g$ the genus of $\Sigma$, $h$ the number of punctures, and $b$ the number of boundary components in the sequel. We always assume the following conditions:
\begin{enumerate}
    \item[(S1)] Each boundary component (if exists) has at least one marked point.
    \item[(S2)] $3(2g-2+h+b)+2|M_\partial| >0$.
    % , where $b$ denotes the number of boundary components.
    \item[(S3)] If $g=0$ and $b=0$, then $h \geq 4$.
\end{enumerate}

% that is, a compact oriented surface with punctures and no boundaries.
% A marked point is called a \emph{puncture} if it lies in the interior of $\Sigma$, and a \emph{special point} otherwise. 
% Let $P=P(\Sigma)$ (resp. $M_\partial=M_\partial(\Sigma)$) denote the set of punctures (resp. special points). A marked surface is called a \emph{punctured surface} if it has empty boundary (and hence $M_\partial=\emptyset$). 
% We denote by $g$ the genus of $\Sigma$, $h$ the number of punctures.
% $b$ the number of boundary components, and $m_i$ the number of marked points on the $i$-th boundary component for $i=1, \dots, b$.
% We always assume the following conditions:
% \begin{enumerate}
%     % \item[(S1)] $b_i >0$.
%     \item $2g-2+h >0$,
%     % , where $b$ denotes the number of boundary components.
%     \item If $g=0$, then $h \geq 4$.
% \end{enumerate}

Let $I = I(\Sigma) := \{1, 2, \dots, 6g-6+3h \}$.
By fixing a vertex $t_0 \in \bT_I$ and an ideal triangulation $\tri_0$ of $\Sigma$ with a labeling $\ell_0: I \to \tri_0$, we obtain the seed pattern $\bs_\Sigma$ associated with $\Sigma$ as follows:
\begin{itemize}
    \item 
    % $N^{(t_0)} = N^{\tri_0} := \bigoplus_{\alpha \in \tri} \bZ e^{\tri_0}_{\alpha}$,
    $B^{(t_0)} := B^{(\tri_0, \ell_0)} := (b_{\ell_0(i), \ell_0(j)}^{\tri_0})_{i,j \in I}$.
    \item $B^{(t)} := B^{(\tri, \ell)}$ for $t \in \bT_I$, where $(\tri, \ell)$ is the labeled triangulation of $\Sigma$ obtained from $(\tri_0, \ell_0)$ by labeled flipping along the path $\gamma: t_0 \to t$.
\end{itemize}
We will abbreviate as
$\mathcal{Z}_{\bs_\Sigma} = \mathcal{Z}_\Sigma$ for $\mathcal{Z} = \cA, \cX$.
A \emph{mapping class} on $\Sigma$ is a homotopy class of an orientation preserving homeomorphism of $\Sigma$.
We can think a mapping class of $\Sigma$ as a mutation loop of the seed pattern $\bs_\Sigma$ as follows:
Let $\phi$ be a mapping class and $t \in \bT_I$ be a vertex with a labeled triangulation $(\tri, \ell)$ such that $B^{(t)} = B^{(\tri,\ell)}$.
Then, the triangulation $\phi^{-1}(\tri)$ is obtained from $\tri$ by flipping along some path $\gamma: t \to t'$.
Thus, there is a permutation $\sigma \in \fS_I$ such that $\sigma. B^{(t')} = B^{\phi^{-1}(\tri, \ell)}$.
Since $B^{\phi^{-1}(\tri,\ell)} = B^{(\tri, \ell)}$, the path $\gamma$ represents a mutation loop and in some sense, it corresponds to the mapping class $\phi$ (see \cite[Theorem 4.5]{IK20a}).

An orientation perserving homeomorphism $f$ of $\Sigma$ is \emph{pseudo-Anosov} if there are two measured foliations $\cF^+$, $\cF^-$ and a real number $\lambda >1$, called stretch factor of $f$, such that $\cF^+ \pitchfork \cF^-$ and $f(\cF^\pm) = \lambda^{\pm 1} \cF^\pm$.
A mapping class is \emph{pseudo-Anosov} if it is represented by a pseudo-Anosov homeomorphism.
For a pseudo-Anosov mapping class $\phi$, its stretch factor $\lambda_\phi$ is defined to be the stretch factor of a pseudo-Anosov representative.

We have the following characterization of the pseudo-Anosov mapping classes on a punctured surface:
\begin{thm}[{\cite[Theorem 1.2]{IK20a}}]
Let $\phi$ be a mapping class of a punctured surface $\Sigma$.
Then, the following conditions are equivalent:
\begin{enumerate}
    \item The mapping class $\phi$ is pseudo-Anosov.
    \item The mutation loop $\phi$ is uniformly sign-stable.
    \item The mutation loop $\phi$ has a North-South dynamics on $\bS \cX_{\Sigma}(\bR^\trop)$ with $\cX$-filling fixed points.
In this case, the cluster stretch factor of the mutation loop $\phi$ coincides with the stretch factor $\lambda_\phi$ of the pA mapping class $\phi$.
\end{enumerate}

\end{thm}

We will enhance the condition (3) of the above theorem.
Namely, we prove the following:
\begin{lem}\label{lem:filling}
The mapping class $\phi$ is pseudo-Anosov if and only if it has non-parabolic North-South dynamics on $\bS \cX_{\Sigma}(\bR^\trop)$ with $\cX$-filling fixed points.
\end{lem}
\begin{proof}
Let $\cF^\pm$ be the pair of measured foliations of the pseudo-Anosov representative of $\phi$.
The condition $\cF^+ \pitchfork \cF^-$ implies that there is some ideal arc $\alpha \in \tri$ of an ideal triangulation $\tri$ of $\Sigma$,
\end{proof}

In order to return our setting, we review the works by Labardini-Fragoso. 
He gives a way to associate a good
\footnote{More precisely, this quiver with potentials are non-degenerate except for some special cases.}
quiver with non-degenerate potential $(Q^\tri, W^\tri)$ for a (tagged) ideal triangulation $\tri$ marked surface $\Sigma$ \cite{LF09}.
Namely, for an ideal triangulation $\tri$ of a marked surface $\Sigma$, the quiver with potentials $(Q^{f_k(\tri)}, W^{f_k(\tri)})$ and $\mu_k(Q^\tri, W^\tri)$ are right equivalent.
Here, $f_k$ denotes a flip of $\tri$ at $k \in \tri$.
Thus, by using this type of quiver with potentials for defining the seed pattern $\bs^\mathrm{cat}$, all mutation loops are QP mutation loops.

By combining with \cref{lem:filling} and \cref{cor:NS_pA_auto}, we get the next corollary.

\begin{cor}
Let $\phi$ be a pseudo-Anosov mapping class of a punctured surface $\Sigma$.
Then since $\phi$ is a sign-stable mutation loop, there is the derived equivalence $F_\phi^{(t)}$ of the Ginzburg dg algebra $\Gamma^{(t)}$ for each $t \in \bT_I$, and we have
\begin{align*}
    h_T(F_\phi^{(t)}|_{\Dfd}) = h_0(F_\phi^{(t)}|_\per)\, (= h_{\rm alg}(\phi^a) = h_{\rm alg}(\phi^x)) = h_{\rm top}(\phi)
\end{align*}
for all $T \in \bR$.
\end{cor}

\paragraph{\textbf{Weak sign stability}}

In \cite{IK20b}, the sign stability of pseudo-Anosov mapping classes of the marked surfaces with possibly non-empty boundaries is discussed.
First, the mapping class $\phi$ on a marked surface $\Sigma$ is pseudo-Anosov if the induced one $\pi(\phi)$ on the punctured surface $\bar{\Sigma}$, which is obtained from $\Sigma$ by each boundary component of $\Sigma$ shrinking to a puncture, is pseudo-Anosov.
The conclusion is the pseudo-Anosovness of this case is characterized by the \emph{weak sign stability} (and some technical conditions).
Namely, it admits the existence of the non-stabilizing components in the stable sign.
Thus, one has to choose the sign of the non-stabilizing components to define the functor $F_\phi$.
The presentation matrices corresponding to two different stable signs are different in general, so the functor $F_\phi$ is not unique since we cannot use \cref{lem:rigid_index}.
Although, for each sign, we can get the inequality of the entropies.

\begin{prop}[{\cite[Proposition 5.2]{IK20b}}]\label{prop:weak_SS_eigen}
Let $\phi$ be a pseudo-Anosov mapping class on a marked surface $\Sigma$ and $\gamma$ be a representation path of $\phi$.
For any sign $\bep \in \{+,-\}^{h(\gamma)}$ such that the stabilizing components of $\bep_\gamma^\stab$ coincide with the corresponding component of $\bep$, the presentation matrix $E_{\gamma,\bep}$ has the stretch factor $\lambda_{\pi(\phi)}$ of the pseudo-Anosov mapping class $\pi(\phi)$ on $\bar{\Sigma}$ as a positive eigenvalue.
\end{prop}

By combining with \cref{thm:K_lower_bound}, we have the following:
\begin{cor}
Let $\phi$, $\gamma$ and $\bep$ as \cref{prop:weak_SS_eigen}.
Then, we define the functor $F_{\gamma, \bep}$ as the composition
\begin{align*}
    F_{\gamma, \bep}: \sD \xrightarrow{(F^*_{\gamma, \bep})^{-1}} \sD' \xrightarrow{\sim} \sD.
\end{align*}
Then, we have
\begin{align*}
    h_0(F_{\gamma, \bep}) \geq \lambda_{\pi(\phi)}.
\end{align*}
\end{cor}

\subsection{Comments on the DHKK pseudo-Anosovness}
In the paper \cite{DHKK}, which initiates the categorical dynamical systems, they already proposed the definition of the pseudo-Anosov autoequivalence, it is called \emph{DHKK pseudo-Anosov} autoequivalence in \cite{FFHKL}.
The DHKK pseudo-Anosov autoequivalences are pseudo-Anosov in the sense of \cite{FFHKL} (\cite[Theorem 2.15]{FFHKL}).
In this subsection, we recall that the geometric interpretation of the space of stability conditions of the finite-dimensional derived category $\Dfd(\Gamma)$ of the Ginzburg dg algebra associated with an ideal triangulation of a punctured surface, and see that the obstruction to assert that the autoequivalence $F_\phi|_{\Dfd}$ is DHKK pseudo-Anosov.

First, we recall the definition of the DHKK pseudo-Anosovness.
\begin{defi}[DHKK pseudo-Anosov, {\cite[Definition 4.1]{DHKK}, \cite[Definition 2.14]{FFHKL}}]
An autoequivalence $F \in \Aut(\sD)$ is \emph{DHKK pseudo-Anosov} if there exists a stability condition $\sigma \in \Stab(\sD)$ and $\lambda >1$ such that
\begin{align}\label{eq:DHKK_pA}
    F \cdot \sigma = \sigma \cdot \begin{pmatrix}\lambda^{-1} & 0 \\ 0 & \lambda\end{pmatrix}.
\end{align}
Here, the matrix $\begin{pmatrix}\lambda^{-1} & 0 \\ 0 & \lambda\end{pmatrix}$ denotes some lift to the universal covering of $GL^+(2, \bR)$.
\end{defi}

One can find some examples of the DHKK pseudo-Anosov autoequivalences.
For instance, on the derived category of the path algebra of a $\ell\geq 3$-Kronecker quiver \cite[Section 4]{DHKK}, on a derived category of an elliptic curve \cite[Proposition 4.14]{Kik19} and on the topological Fukaya category  of a marked surface \cite[Section 4]{FFHKL} (\emph{cf.}, \cite{HKK}).
We note that our setting is quite similar to the setting of \cite[Section 4]{FFHKL} but they are completely different (\emph{cf.}, \cite{IQ}).
% Roughly speaking, the topological Fukaya category is the derived category of the Jacobian algebra of the quiver with potential of a marked surface.

Let $\phi$ be a pseudo-Anosov mapping class of a punctured surface $\Sigma$.
Then, its pseudo-Anosov representatives determine the transverse measured foliations $\cF^+_\phi$ and $\cF^-_\phi$ such that they have the same $k(\geq 3)$-pronged singularities.
On the other hand, a meromorphic quadratic differential $q$ on a Riemann surface $X$ with simple poles at the punctures of $X$ gives such the pair of measured foliations as horizontal and vertical foliations $\cF^h_q$ and $\cF^v_q$ of $q$.
Here, $X$ is a point of the \Teich\ space $\cT(\Sigma)$ of $\Sigma$.
Namely, it is homeomorphic to the surface $\Sigma$.
Let $\cQ(\Sigma)$ denotes the fiber bundle over the \Teich\ space with fibers the space of such the meromorphic quadratic differential on $X \in \cT(\Sigma)$.

\begin{thm}[\cite{HM,Gar}]\label{thm:HM}
The mapping
\[\cF: \cQ(\Sigma) \to \MF(\Sigma) \times \MF(\Sigma)\ ;\ (X, q) \mapsto (\cF^h_q, \cF^v_q)\]
is $MC(\Sigma)$-equivariant and injective.
Here $\MF(\Sigma)$ denotes the space of measured foliations on $\Sigma$.
Moreover, the image of the map $\cF$ is
\[\{ (\cF_1, \cF_2) \mid \mbox{$\mathrm{Int}(\cF_1,\gamma) \neq 0$ or $\mathrm{Int}(\cF_2, \gamma) \neq 0$ for all simple closed curve $\gamma$ on $\Sigma$} \}.\]
\end{thm}
The arationarity of the measured foliations $\cF^\pm_\phi$ of a pseudo-Anosov mapping class $\phi$ implies that the pair $(\cF^+_\phi, \cF^-_\phi)$ contained in the image of the map $\cF$.
Thus, there is a point $(X_\phi, q_\phi) \in \cQ(\Sigma)$ corresponding to $(\cF^+_\phi, \cF^-_\phi)$ by the map $\cF$.
If all the zeroes of quadratic differential $q_\phi$ on the Riemann surface $X_\phi$ are simple, then we will say that $\phi$ is \emph{simple}.

\begin{rmk}
In the space of measured foliations $\MF(\Sigma)$ the degree of pronged singularities is ill-defined since it is not preserved by the Whitehead moves.
However, the degree of pronged singularities of horizontal foliation of the holomorphic quadratic differential $q$ of a Riemann surface $X$ corresponds to the degree of the zeroes of $q$, so it is meaningful.
In this sense, the pseudo-Anosov mapping $\phi$ is simple if and only if the foliation $\cF^+_\phi$ has only $1$- or $3$-pronged singularity.
\end{rmk}

Next, we recall the geometric interpretation of the space of stability conditions of the finite-dimensional derived category $\sD$ of the Ginzburg dg algebra associated with the seed attachment arising from the ideal triangulations of a punctured surface $\Sigma$.

Let $\cQ_2(\Sigma)$ denotes the fiber bundle over the \Teich\ space $\cT(\Sigma)$ with the fiber at $X \in \cT(\Sigma)$ is the space of meromorphic quadratic differentials on $X$ with simple or double poles at the punctures of $X$.
Moreover, we consider the subspace $\cQ^\circ_2(\Sigma) \subset \cQ_2(\Sigma)$ which is consisting of the quadratic differentials with simple zeroes.
One can verify that the residue of a double pole of a meromorphic quadratic differential is well-defined up to a sign.
By choosing the signs of residues at every double poles, we obtain the $2^P$-folded branched cover
\begin{align*}
    \widehat{\cQ}_2(\Sigma) \to \cQ_2(\Sigma)
\end{align*}
with the ramification locus is given by $\cQ(\Sigma)$.
We write $\widehat{\cQ}^\circ_2(\Sigma) \subset \widehat{\cQ}_2(\Sigma)$ for the subspace obtained by the pullback along the covering, and $\cQ^\circ(\Sigma) := \cQ(\Sigma) \cap \cQ_2^\circ(\Sigma)$.

\begin{thm}[{\cite[Theorem 10.3]{All21}}]
Let $\Sigma$ be a punctured surface which is neither a once punctured surface nor a sphere with $\geq 5$ punctures, and $\sD$ denotes the finite-dimensional derived category of the Ginzburg dg algebra associated with the seed pattern of $\Sigma$.
Then, there is an isomorphism as the complex manifolds
\begin{align*}
    \widehat{\cQ}^\circ_2(\Sigma) \cong \Stab^\dagger(\sD)/ST^\dagger(\sD).
\end{align*}
Here, $ST^\dagger(\sD)$ is the group generated by the spherical twists along the simple objects of the heart of a stability condition $\sigma \in \Stab^\dagger(\sD)$ modulo those which are preserving all the heart of the stability conditions in the component $\Stab^\dagger(\sD)$.
Moreover, this isomorphism is equivariant under the actions $MC(\Sigma) \ltimes (\bZ/2)^P \cong Aut^\dagger(\sD)/ST^\dagger(\sD)$ \footnote{
The quotient group $Aut^\dagger(\sD)/ST^\dagger(\sD)$ is naturally isomorphic to the cluster modular group of the seed pattern $\bs_\Sigma$.
}, where $Aut^\dagger(\sD)$ is the group of automorphisms of $\sD$ which preserves the component $\Stab^\dagger(\sD)$ modulo those which are like as in $ST^\dagger(\sD)$.
\end{thm}

By \cite{Vee,BS}, there is a locally-defined map $\pi: \widehat{\cQ}^\circ_2(\Sigma) \to \Hom_\bZ(K_0(\sD), \bC)$ defined by the \emph{period map}, which is local isomorphism.
Roughly speaking, it is equivalent to measuring horizontal and vertical foliations of a quadratic differential by arcs which form a dual graph of an ideal triangulation of $\Sigma$
\footnote{For a quadratic differential which has simple poles, we have to deal with it more carefully \cite[Section 6]{BS}}.
Moreover, there is a locally-defined map $J: \widehat{\cQ}^\circ_2(\Sigma) \to \Stab^\dagger(\sD)/Aut_0^\dagger(\sD)$ it commute the following diagmam:
\[\begin{tikzcd}[column sep=small]
\widehat{\cQ}^\circ_2(\Sigma) \ar[rr, "J"] \ar[rd, "\pi"'] && \Stab^\dagger(\sD)/Aut_0^\dagger(\sD) \ar[ld]\\
& \Hom_\bZ(K_0(\sD), \bC).
\end{tikzcd}\]
Here, $Aut^\dagger_0(\sD)$ is a subgroup of $Aut^\dagger(\sD)$ consisting of the automorphisms which act by the identity on the Grothendieck group $K_0(\sD)$.
% Comparing \cref{thm:HM}, if a pseudo-Anosov mapping class $\phi$ is simple, 

If a pseudo-Anosov mapping class $\phi$ on a punctured surface $\Sigma$ is simple, the stability condition $J(q_\phi)$ behaves like the stability condition of the pseudo-Anosovness of the functor $F_\phi|_{\Dfd}$.
Therefore, the DHKK pseudo-Anosovness of $F_\phi|_{\Dfd}$ is equivalent to the lifting problem of $J(q_\phi)$ to $\Stab^\dagger(\sD)$.
% By this theorem, we can 
\begin{prob}
Are there a lift $\sigma_\phi$ of $J(q_\phi) \in \Stab^\dagger(\sD)/Aut_0^\dagger(\sD)$ to $\Stab^\dagger(\sD)$ satisfying the condition of the DHKK pseudo-Anosovness \eqref{eq:DHKK_pA} ?
\end{prob}

% \begin{thm}
% Let $\phi$ be a simple pseudo-Anosov mapping class of a punctured surface $\Sigma$.
% Then, the functor $F_\phi|_{\Dfd}$ is DHKK pseudo-Anosov.
% \end{thm}

% \begin{thm}
% Let $\widetilde{\cT}(\Sigma)$ denotes the decorated \Teich\ space.
% Then, the cotangent bundle $T^* \widetilde{\cT}(\Sigma)$ of $\widetilde{T}(\Sigma)$ is isomorphic to the space $\mathrm{Quad}_2(\Sigma)$.
% \end{thm}


\begin{thebibliography}{FFHKL19}
\bibitem[All21]{All21}
D. G. L. Allegretti,
\emph{Stability conditions, cluster varieties, and Riemann-Hilbert problems from surfaces},
Adv. Math. \textbf{380} (2021), 107610.

\bibitem[Ami]{Ami}
C. Amiot,
\emph{On Generalized Cluster Categories}, in:
Representations of algebras and related topics, 1--53,
EMS Ser. Congr. Rep., Eur. Math. Soc., Z\"urich, 2011. 

\bibitem[Bon10]{Bon}
M. V. Bondarko,
\emph{Weight structures vs. t-structures; weight filtrations, spectral sequences, and complexes (for motives and in general)},
J. K-Theory, \textbf{6} (2010), 387--504. 

\bibitem[Bri07]{Bri}
T. Bridgeland,
\emph{Stability conditions on triangulated categories},
Ann. of Math., \textbf{166} (2007), 317--345.

\bibitem[BS15]{BS}
T. Bridgeland and I. Smith,
\emph{Quadratic differentials as stability conditions},
Publ. Math. Inst. Hautes \'Etudes Sci., \textbf{121} (2015), 155--278.

\bibitem[DHKK14]{DHKK}
G. Dimitrov, F. Haiden, L. Katzarkov and M. Kontsevich, 
\emph{Dynamical systems and categories},
Contemporary Mathematics, {\bf 621} (2014), 133--170.

\bibitem[DWZ08]{DWZ}
H. Derksen, J. Weyman and A. Zelevinsky, 
{\it Quivers with potentials and their representations I: Mutations},
Selecta Math. (N.S.), {\bf 14} (2008), 59--119.

\bibitem[DWZ10]{DWZ10}
H. Derksen, J. Weyman and A. Zelevinsky,
{\it Quivers with potentials and their representations II: applications to cluster algebras},
J. Amer. Math. Soc., {\bf 23} (2010), 749--790.

\bibitem[Fan18]{Fan18}
Y.-W. Fan,
\emph{Entropy of an autoequivalence on Calabi-Yau manifolds},
Math. Res. Lett., \textbf{25} (2018), 509--519. 

\bibitem[FFHKL19]{FFHKL}
Y.-W. Fan, S. Filip, F. Haiden, L. Katzarkov and Y. Liu
\emph{On pseudo-Anosov autoequivalences},
preprint, arXiv:1910.12350.

\bibitem[FZ02]{FZ-CA1}
S. Fomin and A. Zelevinsky,  
\emph{Cluster algebras. I. Foundations}, 
J. Amer. Math. Soc., \textbf{15} (2002), 497--529. 

\bibitem[FZ07]{FZ-CA4}
S. Fomin and A. Zelevinsky,  
\emph{Cluster algebras. IV. Coefficients}, 
Compos. Math., \textbf{143} (2007), 112--164. 

\bibitem[FF20]{FF}
Y.-W. Fan and S. Filip,
\emph{Asymptotic shifting numbers in triangulated categories},
preprint, arXiv:2008.06159.

\bibitem[FG09]{FG09}
V. V. Fock and A. B. Goncharov, 
\emph{Cluster ensembles, quantization and the dilogarithm},
Ann. Sci. \'Ec. Norm. Sup\'er., \textbf{42} (2009), 865--930.

\bibitem[FLP]{FLP} A. Fathi, F. Laudenbach and V. Poenaru, {\it Travaux de Thurston sur les surfaces}, Asterisque66-67, Soci\'et\'e Math\'ematique de France, Paris, 1979.

\bibitem[Gar]{Gar}
F. Gardiner,
\emph{\Teich\ theory and quadratic differentials},
Wiley Interscience, 1987.

\bibitem[GHKK18]{GHKK}
M. Gross, P. Hacking and S. Keel and M. Kontsevich,
\emph{Canonical bases for cluster algebras},
J. Amer. Math. Soc., \textbf{31} (2018), 497--608. 

\bibitem[GMN13]{GMN}
D. Gaiotto, G. Moore and A. Neitzke, 
\emph{Wall-crossing, Hitchin systems, and the WKB approximation},
Adv. Math., \textbf{234} (2013), 239--403.

\bibitem[Gin06]{Gin}
V. Ginzburg,
\emph{Calabi-Yau algebras},
preprint, arXiv:0612139.

\bibitem[HKK17]{HKK}
F. Haiden, L. Katzarkov and M. Kontsevich,
\emph{Flat surfaces and stability structures},
Publ. Math. Inst. Hautes \'Etudes Sci. \textbf{126} (2017), 247--318.

\bibitem[HM79]{HM}
J. Hubbard and H. Masur,
\emph{Quadratic differentials and foliations},
Acta Math. \textbf{142} (1979), 221--274.

\bibitem[Ike21]{Ike} A. Ikeda, {\it Mass growth of objects and categorical entropy},
to appear in Nagoya Math. J. 
% \href{https://arxiv.org/abs/1612.00995v1}{arXiv:1612.00995}.

\bibitem[IQ18]{IQ}
A. Ikeda, Y. Qiu,
\emph{$q$-Stability conditions on Calabi-Yau-$\mathbb{X}$ categories and twisted periods},
preprint, arXiv:1812.00010.

\bibitem[IIKKN13]{IIKKN}
R. Inoue, O. Iyama, B. Keller, A. Kuniba and T. Nakanishi,
\emph{Periodicities of T-
systems and Y-systems, dilogarithm identities, and cluster algebras I: type $B_r$,}
Publ. Res. Inst. Math. Sci. \textbf{49} (2013), 1--42.

\bibitem[IK21]{IK19}
T. Ishibashi and S. Kano,
{\em Algebraic entropy of sign-stable mutation loops},
to appear in Geometriae Dedicata,
arXiv:1911.07587.

\bibitem[IK20a]{IK20a}
T. Ishibashi and S. Kano,
{\em Sign stability of mapping classes on marked surfaces I: empty boundary case},
preprint, arXiv:2010.05214.

\bibitem[IK20b]{IK20b}
T. Ishibashi and S. Kano,
{\em Sign stability of mapping classes on marked surfaces II: general case via reductions},
preprint, arXiv:2011.14320.

\bibitem[Kik17]{Kik17}
K. Kikuta,
\emph{On entropy for autoequivalences of the derived category of curves},
Adv. Math., \textbf{308} (2017), 699--712.

\bibitem[Kik19]{Kik19}
K. Kikuta,
\emph{Curvature of the space of stability conditions},
preprint, arXiv:1907.10973.

\bibitem[KST20]{KST20}
K. Kikuta, Y. Shiraishi and A. Takahashi,
\emph{A note on entropy of auto-equivalences: lower bound and the case of orbifold projective lines},
Nagoya Math. J., \textbf{238} (2020), 86--103.

\bibitem[KT19]{KT19}
K. Kikuta and A. Takahashi,
\emph{On the categorical entropy and the topological entropy},
Int. Math. Res. Not. IMRN, \textbf{2019} (2019), 457--469.

\bibitem[KY11]{KY}
B. Keller and D. Yang,
{\it Derived equivalences from mutations of quivers with potential},
Adv. Math., {\bf 226} (2011), 2118--2168.

\bibitem[Kel1]{Kel1}
B. Keller,
\emph{Cluster algebras, quiver representations and triangulated categories},
London Math. Soc. Lecture Note Ser., \textbf{375}, Cambridge Univ. Press, Cambridge, 2010, 76--160.

\bibitem[Kel2]{Kel2}
B. Keller,
\emph{Cluster algebras and derived categories},
EMS Ser. Congr. Rep., Eur. Math. Soc., Z\"urich, 2012, 123--183.

\bibitem[LF09]{LF09}
Daniel Labardini-Fragoso,
{\it Quivers with potentials associated to triangulated surfaces.}
Proc. Lond. Math. Soc. (3) {\bf 98} (2009), 797--839.

\bibitem[Lad12]{Lad12}
S. Ladkani,
\emph{On Jacobian algebras from closed surfaces},
preprint, arXiv:1207.3778

\bibitem[Nag13]{Nag}
K. Nagao,
\emph{Donaldson--Thomas theory and cluster algebras},
Duke Math. J., \textbf{162} (2013), 1313--1367. 

\bibitem[NZ12]{NZ12}
T. Nakanishi and A. Zelevinsky, 
\emph{On tropical dualities in cluster algebras},
Algebraic groups and quantum groups, Contemp. Math., Amer. Math. Soc., Providence, RI, \textbf{565} (2012), 217--226.

\bibitem[Ou18]{Ou18}
G. Ouchi,
\emph{Automorphisms of positive entropy on some hyperK\"ahler manifolds via derived automorphisms of K3 surfaces},
Adv. Math., \textbf{335} (2018), 1--26.

\bibitem[Ou20]{Ou20}
G. Ouchi,
\emph{On entropy of spherical twists},
With an appendix by Arend Bayer,
Proc. Amer. Math. Soc., \textbf{148} (2020), no. 3, 1003--1014.

\bibitem[Pla10a]{Plaa}
P. -G. Plamondon,
\emph{Cluster algebras via cluster categories with infinite-dimensional morphism spaces},
Compositio Mathematica, \textbf{147} (2011), 1921--1954.

\bibitem[Pla10b]{Plab}
P. -G. Plamondon,
\emph{Cluster characters for cluster categories with infinite-dimensional morphism spaces},
Adv. Math., \textbf{227} (2011), 1--39.

\bibitem[TS94]{TS}
B. S. Tam and H. Schneider,
\emph{On the core of a cone-preserving map},
Trans. Amer. Math. Soc., \textbf{343} (1994), 479--524. 

\bibitem[Th88]{Th88} W. Thurston, {\it On the geometry and dynamics of diffeomorphisms of surfaces}, Bull. Amer. Math. Soc. (N.S.), {\bf 19} (1988), 417--431.

\bibitem[Vee90]{Vee}
W. Veech,
\emph{Moduli spaces of quadratic differentials},
J. Analyse Math., \textbf{55} (1990), 117--171.

\bibitem[Wo12]{Wo}
J. Woolf,
\emph{Some metric properties of spaces of stability conditions},
Bull. Lond. Math. Soc., \textbf{44} (2012), 1274--1284.

\bibitem[Yo20]{Yo20}
K. Yoshioka,
\emph{Categorical entropy for Fourier--Mukai transforms on generic abelian surfaces}
J. Algebra, \textbf{556} (2020), 448--466.

\end{thebibliography}
\end{document}